\documentclass[12pt,reqno]{amsart}
\UseRawInputEncoding
\usepackage{amsthm}
\usepackage{color}
\usepackage{amssymb, amsmath, amsfonts, latexsym}
\usepackage{enumerate}
\usepackage{bbm}
\usepackage{hyperref}

\usepackage{multimedia}

\newtheorem{thm}{Theorem}[]
\newtheorem{lem}[thm]{Lemma}

\newtheorem{rmk}[thm]{Remark}
\theoremstyle{definition}

\numberwithin{equation}{section} \theoremstyle{remark}

\title[Nonstandard limits for $\beta$-Jacobi ensembles]{\bf Nonstandard limit theorems and large deviation for $\beta$-Jacobi ensembles with a different scaling}  
\author[Y-T. Ma]{Y\MakeLowercase{utao} M\MakeLowercase{a}}
\thanks{ Yutao Ma: School of Mathematical Sciences, Beijing Normal University, 100875 Beijing, China. \\ E-mail: mayt@bnu.edu.cn}

\author[Y-H. Mao]{Y\MakeLowercase{ong}-H\MakeLowercase{ua} M\MakeLowercase{ao}}
\thanks{ Yong-Hua Mao: School of Mathematical Sciences, Beijing Normal University, 100875 Beijing, China. \\ E-mail: maoyh@bnu.edu.cn}

\author[S-Y. Wang]{S\MakeLowercase{iyu} W\MakeLowercase{ang}*}
\thanks{*Correspongding author.  Siyu Wang: School of Mathematical Sciences, Beijing Normal University, 100875 Beijing, China.  E-mail: wang\_siyu@mail.bnu.edu.cn}
\thanks{The research of Yutao Ma and Yong-Hua Mao is supported in part by NSFC 12171038, 11871008, the National Key R$\&$D Program of China (No. 2020YFA0712900) and 985 Projects. }

\begin{document}

\begin{abstract}
We consider $\beta$-Jacobi ensembles with parameters $p_1, p_2\geq n.$ We prove that the empirical measure of the rescaled Jacobi ensembles converges weakly to a modified Watcher law via the spectral measure method, which revisits the weak limits obtained in \cite{MaLDPJ} while replacing the condition $\beta n\!>\!> \log n$ by $\beta n\!>\!>1.$     
  We also provide the central limit theorem and the large deviation for the corresponding rescaled spectral measure.
\end{abstract}
\maketitle

{\bf AMS Classification Subjects 2020:} 60F10, 15B52

\section{introdution } 
A $\beta$-Jacobi ensemble, also called the $\beta$-MANOVA ensemble, is a random vector on $[0, 1]^n$ whose joint density function is  given by 
\begin{equation}\label{dfjacobi}
	f_{\beta, p_1, p_2}(x_1, \cdots, x_n)=c_{\beta, p_1, p_2}\prod_{1\le i<j \le n}|x_i-x_j|^{\beta} \prod_{i=1}^n x_i^{\frac{\beta}{2}(p_1-n+1)-1 } (1-x_i)^{\frac{\beta}{2}(p_2-n+1)-1 } 
\end{equation} 
with parameters $p_1, p_2\geq n\geq 1$ and $\beta>0$ and $c_{\beta, p_1, p_2}$ being the normalizing constant. 

Let $\boldsymbol{\lambda}=(\lambda_1, \cdots, \lambda_n)$ be the $\beta$-Jacobi ensembles with joint density function \eqref{dfjacobi}. 
When $\beta=1, 2$ and $4,$ the function \eqref{dfjacobi} can be regarded as the joint density function of the eigenvalues of a particular random matrix 
 $$A=(X X^*+YY^*)^{-\frac12}X X^*(X X^*+YY^*)^{-\frac12}$$ 
 with $X,Y$ are $n\times (\beta p_1)$ and $n\times \, (\beta p_2) (p_1, p_2\geq n)$ matrices with i.i.d. real ($\beta=1$), complex ($\beta=2$) or quaternion ($\beta=4$) standard normal distributed entries. The spectral properties of the $\beta$-Jacobi ensemble can be applied to multivariate analysis of variance (see \cite{Muihead}). In statistical mechanics, the $\beta$-Jacobi ensemble arises as a model of log gas with logarithmic interactions, confined to the interval $[0, 1]$ (see \cite{Dyson62}, \cite{Forrester}). One central object in the asymptotic study of random matrix ensembles is the empirical measures. 
  
 There are different ways to obtain the weak limits of the empirical measures. See Johansson \cite{Johnsson} for an approach directly based on the joint density function, Dumitriu and Edelman \cite{DE06} for a combinatorial method based on the random tridiagonal matrix models, see also alternative ways in the books \cite{AnG} and \cite{PS}. In this paper,  we follow the spectral measure method in \cite{Dette} (see also \cite{Nagel}, \cite{Trinh}, \cite{TT-21}) for the weak limit of the empirical measures.  
 In \cite{Nagel}, they offered the Marchenko-Pastur law and the semicircle law as the weak limits of the empirical measures for different scaling $\lambda_i$ via the spectral measure method when $p_1$ and $p_2$ depend on $n$ and $\beta>0$ fixed. Furthermore, if $n/p_1\to 0$ and $p_1/p_2 \to \sigma >0,$ they showed that the empirical measure of $$\left((1+\sigma)\sqrt{\frac{(1+\sigma)p_2}{\sigma n}}\left(\lambda_i-\frac{\sigma}{\sigma+1}\right)\right)_{1\le i\le n}$$ converges weakly to the semicircle law with probability one under an additional condition 
\begin{align}\label{condsemi}
	\sqrt{\frac{ p_2}{n}}\left(\sigma-\frac{p_1}{p_2}\right)\to 0.
\end{align}
By contrast, under the conditions $n/p_1\to \gamma\in (0, 1]$ and $ p_1/p_2 \to 0 ,$ the empirical measure of $(p_2\lambda_i/p_1)_{1\le i\le n}$ converges weakly to the Marchenko-Pastur law with probability one. 

In \cite{MaLDPJ}, setting $Y_i=((p_1+p_2)\lambda_i-p_1)/\sqrt{n p_1}$ for $1\le i\le n$ and $L_n:=\frac1n\sum_{i=1}^n\delta_{Y_i},$ they established the large deviation for $L_n$ when $$p_1/p_2\to \sigma\in [0, +\infty), \quad n/p_1\to\gamma \in [0, 1]$$ and $\beta n\!>\!>\log n.$ As a consequence, they proved that $L_n$ converges weakly to $\tilde{\nu}_{\gamma, \sigma}, $ whose density function is given by 
\begin{equation}\label{dens}\tilde{h}_{\gamma, \sigma}(x)= \begin{cases} \frac{\sigma\sqrt{\gamma}}{1+\sigma}h_{\gamma, \sigma}\left(\frac{\sigma}{1+\sigma}(\sqrt{\gamma} x+1)\right),  & 0<\sigma\gamma\le 1\\
\sqrt{\gamma} h_{\gamma}(1+\sqrt{\gamma} x),  & \sigma=0,  0< \gamma\le 1; \\
\frac{1+\sigma}{2\pi}\sqrt{\frac{4}{1+\sigma}-x^2}, & \gamma=0, \sigma\geq 0.\\
  \end{cases}\end{equation} 
  Here $h_{\gamma}$ is the density function of the well-known Marchenko-Pastur law $\nu_{\gamma}$ and $h_{\gamma, \sigma}$ is the density function of the Wachter law $\nu_{\gamma, \sigma},$ which are given as follows 
  \begin{equation}\label{hgam} h_{\gamma, \sigma}(x)=\frac{1+\sigma}{2\pi\sigma\gamma}\frac{\sqrt{(x-u_1)(u_2-x)}}{x(1-x)}1_{[u_1, u_2]}(x); \end{equation} 
$$h_{\gamma}(x)=\frac{1}{2\pi\gamma x}\sqrt{(x-\gamma_{1})(\gamma_{2}-x)}1_{\gamma_1\le x\le\gamma_2}, $$  
 where $\gamma_1=(\sqrt{\gamma}-1)^2, \, \gamma_2=(\sqrt{\gamma}+1)^2$ and $$u_1=\frac{\sigma}{1+\sigma}\left(\sqrt{1-\frac{\sigma\gamma}{1+\sigma}}-\sqrt{\frac{\gamma}{1+\sigma}}\right)^2, \quad u_2=\frac{\sigma}{1+\sigma}\left(\sqrt{1-\frac{\sigma\gamma}{1+\sigma}}+\sqrt{\frac{\gamma}{1+\sigma}}\right)^2.$$
The weak limit obtained in \cite{MaLDPJ} unified and improved the related limits in \cite{Nagel}, however with an additional condition $\beta n \!>\!>\log n$. 
One main object of this paper is to revisit the weak limits of $\mu_n$ in \cite{MaLDPJ} via spectral measure method, which can help to relax the limitation $\beta n\!>\!>\log n.$ This will be done via the weak limit of the corresponding spectral measure. Here comes our main result. 

\begin{thm}\label{main}
	Let $\boldsymbol{\lambda}=(\lambda_1, \cdots, \lambda_n)$ be the $\beta$-Jacobi ensembles with joint density function \eqref{dfjacobi} and set $Y_i=((p_1+p_2)\lambda_i-p_1)/\sqrt{n p_1}$
	for $1\le i\le n$ and $L_n=\frac{1}{n}\sum_{i=1}^n \delta_{Y_i}.$ 
	Suppose 
$$({\bf H}): \qquad\lim\limits_{n\to\infty}\frac{n}{p_1}=\gamma\in [0, 1], \quad \lim\limits_{n\to\infty}\frac{p_1}{p_2}=\sigma\in [0, \gamma^{-1}), \quad \text{and} \quad \lim_{n\to\infty}\beta n=+\infty.$$		
Then $L_{n}$ converges weakly to the probability  $\widetilde \nu_{\gamma,\sigma}$  almost surely, and the density function of $\tilde{\nu}_{\gamma, \sigma}$ is given in \eqref{dens}. 
 \end{thm}

\begin{rmk} This limit improves the condition $\beta n\!>\!> \log n$ in \cite{MaLDPJ} to $\beta n\!>\!>1.$ When $\gamma\sigma\in (0, 1],$ the same method in the proof of Theorem \ref{main} leads that $\frac{1}{n}\sum_{i=1}^n\delta_{p_2\lambda_i/p_1}$ converges weakly to the Wachter law $\nu_{\gamma, \sigma}$ almost surely. Moreover, when $\gamma>0$ while $\sigma=0,$ 
$\frac{1}{n}\sum_{i=1}^n\delta_{p_2\lambda_i/p_1}$ converges weakly to the Marchenko-Pastur law almost surely, as proved in \cite{Nagel}. While for $\gamma=0,$ $L_n$ converges weakly to a semicircle law $\frac{1+\sigma}{2\pi}\sqrt{\frac{4}{1+\sigma}-x^2}$ almost surely while a similar result was obtained in Theorem 2.1 of \cite{Nagel} with an additional conditions \eqref{condsemi} because of different scaling.  
\end{rmk}

\begin{rmk}[\bf The relationship between the three classical laws] 
As pointed out in \cite{MaLDPJ}, under condition {\bf H}, $$\lim_{\sigma\to 0}W_2(\tilde{\nu}_{\gamma, \sigma}, \tilde{\nu}_{\gamma})=0, \quad \lim_{\sigma\to 0}W_2(\tilde{\nu}_{\gamma, \sigma}, SC(2/(1+\sigma)))=0 \quad {\rm and} \quad \lim_{\gamma\to 0}W_2(\tilde{\nu}_{\gamma}, SC(2))=0.$$ 
Here the density functions of $\tilde{\nu}_{\gamma}$ and $SC(\alpha)$ are $\sqrt{\gamma} h_{\gamma}(1+\sqrt{\gamma} x)$ and  $\frac{1}{\pi\alpha}\sqrt{2\alpha-x^2},$ respectively. 
We offer a graphical  interchange format (GIF) to show the fluctuations of the density function $\widetilde{h}_{\gamma, \sigma}$ as $\gamma$ and $\sigma$ vary. The GIF and corresponding R code  can be downloaded from 
\href{https://github.com/S1yuW/Nonstandard\_limits\_for\_beta\_Jacobi\_ensembles}{https://github.com/S1yuW/Nonstandard\_limits\_for\_beta\_Jacobi\_ensembles}.
\end{rmk}

Based on the weak limit of the empirical measure and    Gershgorin circle Theorem, the same argument as in \cite{MaLDPJ} deduces the following strong law of large numbers of the extremal eigenvalues.  
\begin{rmk}  Let $\boldsymbol{\lambda}$ be the $\beta$-Jacobi ensembles with joint density function \eqref{dfjacobi} and 
set $\lambda_{(1)}=\min_{1\le i\le n}\lambda_i$ and $\lambda_{(n)}=\max_{1\le i\le n}\lambda_i$. Under assumption {\bf H}, we have 
$$\aligned \frac{(p_1+p_2)\lambda_{(1)}-p_1}{\sqrt{np_1}}&\to\frac{(1-\sigma)\sqrt{\gamma}-2\sqrt{1+\sigma-\sigma\gamma}}{1+\sigma} \\
\frac{(p_1+p_2)\lambda_{(n)}-p_1}{\sqrt{np_1}}&\to\frac{(1-\sigma)\sqrt{\gamma}+2\sqrt{1+\sigma-\sigma\gamma}}{1+\sigma}
\endaligned  $$ 
almost surely as $n\to\infty.$ These limits are equivalent to the statements of Theorem 2 in \cite{MaLDPJ}, with weaker condition $\beta n\to\infty.$
	The other limiting behaviors under different scaling and the comparison with existing results in literature were also discussed in \cite{MaLDPJ}. Here we don't mention that again.  
	\end{rmk} 

In \cite{MaLDPJ}, they established a full large deviation for the empirical measure $L_n$ with speed $\beta n^2$ and a good rate function. Partial cases of the large deviation of spectral measures were investigated in \cite{Nagel}. Therefore, another object of this paper is to establish the large deviation for the spectral measure under condition {\bf H}. We will follow the method in \cite{Nagel} to achieve this goal. 
 
In the remainder of this paper, we are going to prove Theorem \ref{main} in the second section and the third section is devoted to the central limit theorem and large deviation of the spectral measure. 

\section{Proof of Theorem \ref{main}.} 
We first recall the definition of spectral measure (see \cite{Dette} and \cite{Trinh2}). 
\subsection{Spectral measure} 
Let $J$ be a semi-infinite Jacobi matrix, which is in fact a symmetric tridiagonal matrix of the form 
\begin{align}\label{defj}
	J=\left(\begin{array}{llll}a_{1} & b_{1} & & \\ b_{1} & a_{2} & b_{2} & \\ & \ddots & \ddots & \ddots\end{array}\right), \quad  a_{i} \in \mathbb{R}, \;  b_{i}>0.
\end{align}
For the Jacobi matrix, there exists a  probability measure $\mu$ such that $$
\left\langle\mu, x^{k}\right\rangle=\int_{\mathbb{R}} x^{k} d \mu=e_{1}^TJ^{k} e_{1}=J^k(1, 1) , \; k \in \mathbb{N},
$$
where $e_{1}=(1,0, \ldots)^{T}. $ Here, we use the notation $J(i, j)$ to represent the $(i, j)$ entry of the matrix $J.$ The probability $\mu$ determined by the equation above is unique if and only if $J$ is essentially self-adjoint in $l^2.$ A simple sufficient condition for this is $\sum\limits_{i=1}^{\infty}b_i^{-1}=\infty$ (see \cite{B. Simon}).  In the case of uniqueness, we call $\mu$ the spectral measure of $J$.

The spectral measure of a finite Jacobi matrix, a symmetric tridiagonal matrix of the form 
\begin{align}\label{defjn}
J_n=\left(\begin{array}{cccc}a_{1} & b_{1} & & \\ b_{1} & a_{2} & b_{2} & \\ & \ddots & \ddots & \ddots \\ & & b_{n-1} & a_{n}\end{array}\right), \; a_{i} \in \mathbb{R}, \; b_{i}>0
\end{align}
is defined to be the unique probability measure $\mu$ on $\mathbb{R}$ 
satisfying 
$$\left\langle\mu, x^{k}\right\rangle=e_{1}^T J_n^{k} e_{1}=J_n^{k}(1, 1) , \; k \in \mathbb{N} ,$$
where $e_{1}=(1,0, \ldots)^{T} $.

\subsection{Weak limit of spectral measure} 

The tridiagonal random matrix characterization for $\beta$-Jacobi ensembles was found by Edelman and Sutton in \cite{ES08}. We use the following construction of  Killip and Nenciu \cite{Killip}. Let $c_{1}, \ldots, c_{n}, s_1, \ldots,  s_{n-1}$  be independent random variables distributed as
\begin{equation}\label{cksk}
	\begin{cases}
c_{k} \sim \operatorname{Beta}\left( \beta^{\prime} \left(p_1-k+1  \right) , \beta^{\prime} \left(p_2-k+1 \right) \right), &1\leq k \leq n ; \\
s_{k} \sim \operatorname{Beta}\left(\beta^{\prime} \left(n-k \right)  , \beta^{\prime} \left(p_1+p_2-n-k+1  \right)  \right), &1\leq k \leq n-1;
\end{cases}
\end{equation}
with $\beta'=\beta/2$ and define
$$
\begin{aligned}
&a_{k}=s_{k-1}\left(1-c_{k-1}\right)+c_k\left(1-s_{k-1}\right), \\
&b_{k}=\sqrt{c_k(1-c_k)s_k(1-s_{k-1})},
\end{aligned}
$$
with $c_{0}=s_{0}=0$. Here $\operatorname{Beta}(x, y)$ denotes the beta distribution with parameters $x, y$. Consider the tridiagonal matrix
$$
\mathcal{J}_{n}=\left(\begin{array}{cccc}
a_{1} & b_{1} & & \\
b_{1} & a_{2} & \ddots & \\
& \ddots & \ddots & b_{n-1} \\
& & b_{n-1} & a_{n}
\end{array}\right).
$$
The eigenvalues of $\mathcal{J}_{n}$ possess the joint density function \eqref{dfjacobi} .

Let $\mu_n$ be the unique spectral measure of $$\widetilde{\mathcal{J}}_{n}:=\frac{(p_1+p_2)\mathcal{J}_n-p_1 \bf \, I_n}{\sqrt{n \, p_1}}$$ with ${\bf I}_n$ the $n\times n$ identity matrix and recall $L_n=\frac{1}{n}\sum_{i=1}^n\delta_{Y_i}$ with $(Y_i)_{1\le i\le n}$ 
being the eigenvalues of $\widetilde{\mathcal{J}}_n.$ 
It was proved in \cite{Nagel} that $\mu_n$ and $L_n$ have the same weak limit since their Kolmogorov distance goes to $0.$ 
When $\beta$ is replaced by $\beta_n,$ the condition $\beta_n\to\infty$ certainly guarantees the crucial estimate in the proof of Theorem 4.2 in \cite{Nagel}, and then Theorem \ref{main} can be obtained by the following theorem. 

\begin{thm}\label{spe} Let $\mu_n$ be the spectral measure of $\widetilde{\mathcal{J}}_n$ introduced above. Under the assumption {\bf H}, $\mu_n$ converges weakly to $\tilde{\nu}_{\gamma, \sigma}$ almost surely as $n\to\infty.$  
Furthermore, if $f$ is a continuous function of polynomial growth, then $$ \left\langle\mu_{n}, f\right\rangle \rightarrow \langle \widetilde \nu_{\gamma,\sigma}, f\rangle $$ almost surely and in $L^{q}$ for all $1 \leq q<\infty$. 
\end{thm}

The strategy of spectral measure method to prove Theorem \ref{spe} (see \cite[Th 2.3 and Th 2.9]{Trinh}) is to obtain the entrywise limiting matrix $\widetilde{\mathcal{J}}_{\gamma, \sigma}$ of $\widetilde{\mathcal{J}}_n$ and $\widetilde{\mathcal{J}}_{\gamma, \sigma}$  is expected to be a deterministic Jacobi matrix, whose spectral measure is unique and can be verified as $\tilde{\nu}_{\gamma, \sigma}.$ Thereby, we have two targets: 
to obtain the matrix $\widetilde{\mathcal{J}}_{\gamma, \sigma}$ and to verify that $\tilde{\nu}_{\gamma, \sigma}$ is the unique spectral measure of $\widetilde{\mathcal{J}}_{\gamma, \sigma}$.  
For these two aims, we present two lemmas respectively.   

The first one is on $\operatorname{Beta}$ distributions. Here recall that $\operatorname{Beta}(x, y)$ denotes the beta distribution with parameters $x, y$, and $\chi_k^2$ denotes the chi-squared distribution with $k$ degrees of freedom. 
\begin{lem}\label{beta}
			Let $\left\{x_{n}\right\}$ and $\left\{y_{n}\right\}$ be two sequences of positive real numbers satisfying 
$$
\lim_{n\to\infty}x_n=\lim_{n\to\infty} y_n=+\infty \quad \text{and} \quad \lim_{n\to\infty}\frac{x_{n}}{y_{n}} =\lambda \in[0,+\infty].
$$ 
Then the following limiting behaviors hold.
\begin{itemize}
	\item[(1)]  As $n \rightarrow \infty$, $$\operatorname{Beta}\left(x_{n}, y_{n}\right) \rightarrow \frac{\lambda}{\lambda+1}\mathbf 1_{\{0 \leq \lambda <+\infty\}}+\mathbf 1_{\{\lambda =+\infty\}}  $$ almost surely and in $L_{q}$ for all $q<\infty$ provided $\sum_{n=1}^{\infty}e^{-(x_n\wedge y_n)}<\infty.$ 
	\item[(2)] As $n\to\infty,$ we have 
	$$\frac{x_n+y_n}{l_n \sqrt{x_n}}\left({\rm Beta}(x_n, y_n)-\frac{x_n}{x_n+y_n}\right)\to 0$$
	almost surely and in $L^q$ provided that there exists $k_0$ such that $\sum_{n=1}^{\infty}1/l_{n}^{k_0}$ converges.
	\item[(3)] As $n \rightarrow \infty$,
\begin{align}\label{cltbeta}
	\sqrt{\frac{(x_n+y_n)^3}{x_n y_n}} \left(\operatorname{Beta}\left(x_{n}, y_{n}\right)-\frac{x_{n}}{x_{n}+y_{n}}\right) \to \mathcal{N}\left(0, 1 \right)
\end{align}
weakly.
\end{itemize}
\end{lem}

\begin{proof}
	It was proved in \cite{Nagel} that  
	$$\mathbb{P}\left(\left|{\rm Beta}(x_n, y_n)-\frac{x_n}{x_n+y_n}\right|\ge\varepsilon\right)\le \exp\left\{-\frac{\varepsilon^2}{128}\frac{x_n^3+y_n^3}{x_n y_n}\right\}.  $$
 By condition, we know 
$$\sum_{n=1}^{\infty}\mathbb{P}\left(\left|{\rm Beta}(x_n, y_n)-\frac{x_n}{x_n+y_n}\right|\ge\varepsilon\right)<\infty.$$ 
Hence Borel-Cantelli Lemma guarantees $$
  {\rm Beta}\left(x_{n}, y_{n}\right)-\frac{x_n}{x_n+y_n}\to 0 
$$ almost surely as $n\to\infty.$ This is equivalent to say 
$${\rm Beta}\left(x_{n}, y_{n}\right)\to \frac{\lambda}{\lambda+1}\mathbf 1_{\{0 \leq \lambda <+\infty\}}+\mathbf 1_{\{\lambda =+\infty\}}  $$ 
almost surely as $n\to\infty.$  This implies that convergence in probability holds. Therefore, the convergence in $L_{q}$ is clear because the $q$-th moment of ${\rm Beta}\left(x_{n}, y_{n}\right)$ are bounded by 1.

Now we prove the second item. Observe that 
$$\mathbb{E}\left(\frac{(x_n+y_n)^q}{l_n^q x_n^{q/2}}\left|{\rm Beta}(x_n, y_n)-\frac{x_n}{x_n+y_n}\right|^q\right)=O\left(l_n^{-q}\right)$$
for $q\geq 1.$
Therefore, $$\aligned 
\mathbb{P}\left(\frac{x_n+y_n}{l_n\sqrt{ x_n}}\left|{\rm Beta}(x_n, y_n)-\frac{x_n}{x_n+y_n}\right|\geq \varepsilon\right)&\le \frac{(x_n+y_n)^q}{\varepsilon^q x_n^{q/2} l_n^q}\mathbb{E}\left|{\rm Beta}(x_n, y_n)-\frac{x_n}{x_n+y_n}\right|^q\\
&=O\left(l_n^{-q}\right).
\endaligned $$
By condition on $l_n$,  
$$\sum_{n=1}^{\infty}\mathbb{P}\left(\frac{x_n+y_n}{l_n \sqrt{x_n}}\left|{\rm Beta}(x_n, y_n)-\frac{x_n}{x_n+y_n}\right|\geq \varepsilon\right)<\infty.$$ Borel-Cantelli Lemma again ensures the almost surely convergence. The convergence in $L^q$ holds since the expectation 
$$\mathbb{E}\left(\frac{(x_n+y_n)^q}{l_n^q x_n^{q/2}}\left|{\rm Beta}(x_n, y_n)-\frac{x_n}{x_n+y_n}\right|^q\right)=O\left(l_n^{-q}\right)$$
tending to zero as $n\to\infty.$ 
At last, we prove the fourth item. 
 Define 
$$ \xi_n = \frac{\chi_{2x_n}^{2}}{2\sqrt{x_n}}- \sqrt{x_n}    \quad  \text{ and } \quad \eta_n = \frac{\chi_{2y_n}^{2}}{2\sqrt{y_n}}- \sqrt{y_n}. $$ 
It is known that 
both $\xi_n$ and $\eta_n$ converges weakly to $\mathcal{N}(0, 1)$ as $n\to\infty.$
Then, \begin{align*}
	\operatorname{Beta}\left(x_{n}, y_{n}\right)-\frac{x_{n}}{x_{n}+y_{n}} \stackrel{\rm d}{=} & \frac{\chi_{2 x_{n}}^{2}}{\chi_{2 x_{n}}^{2}+\chi_{2 y_{n}}^{2}}-\frac{x_{n}}{x_{n}+y_{n}} \\
= &\frac{2 \sqrt{x_n}\, y_n \,\xi_k - 2 \sqrt{y_n}\, x_n \, \eta_n }{\left(\chi_{2 x_{n}}^{2}+\chi_{2 y_{n}}^{2}\right)\left(x_{n}+y_{n}\right)}.
\end{align*}
We write the left hand side ({\rm LHS}) of  \eqref{cltbeta} as
$$\text{LHS} = \left(\sqrt{\frac{y_n}{x_n+y_n}}\xi_n - \sqrt{\frac{x_n}{x_n+y_n}}\eta_n \right) \left( \frac{2x_n+2y_n}{\chi_{2 x_{n}}^{2}+\chi_{2 y_{n}}^{2}}  \right),$$ 
which converges weakly to $\mathcal{N}(0, 1)$ as $n\to\infty$ because the factor $ (2x_n+2y_n)/(\chi_{2 x_{n}}^{2}+\chi_{2 y_{n}}^{2})  $ converges almost surely to $1$ and $\xi_n$ and $\eta_n$ are independent and both $\xi_n$ and $\eta_n$ converge weakly to $\mathcal{N}(0, 1)$.  
\end{proof}

The second lemma is on the spectral measure of a deterministic Jacobi matrix. 

\begin{lem}\label{Jacobimu}
	 Let $ \alpha_0, \alpha_1, \beta_0 ,\beta_1$ be real parameters. 
	 Consider the following tridiagonal matrix $J$ with entries satisfying  $$J(1, 1)=\alpha_0, \; J(1, 2)=\beta_0, \; J(k, k)=\alpha_1, \; J(k, k+1)=\beta_1$$ for any $k\geq 2.$ The spectral measure $\mu$ of $J$ has the probability density function  $$ h(x)=\frac{1}{\pi} \frac{\beta_0^2 \sqrt{ 4 \beta_{1}^{2}-\left(x-\alpha_{1}\right)^{2} }}{2 \beta_{0}^{4}+2 \beta_{1}^{2}\left(x-\alpha_{0}\right)^{2}+2 \beta_{0}^{2}\left(x-\alpha_{0}\right)\left(\alpha_{1}-x\right)}.$$ 
	 In particular, by taking $$\alpha_0=0, \; \beta_0=\frac{1}{\sqrt{1+\sigma}}, \;  \alpha_1=\frac{\sqrt{\gamma}(1-\sigma)}{1+\sigma}, \; \beta_1=\frac{\sqrt{1+\sigma-\sigma \gamma}}{1+\sigma}, $$ the corresponding spectral measure $\mu$ is nothing but $\tilde{\nu}_{\gamma, \sigma}.$
\end{lem}

\begin{proof}
Since $J$ is a Jacobi matrix with bounded coefficients, its corresponding spectral measure $\mu$ is unique and has compact support. Let $S_{\mu}$ be the Stieltjes transform of $\mu,$
$$S_{\mu}(z)=\int \frac{1}{x-z}d\mu(x).$$ In the theory of Jacobi matrices, the Stieltjes transform $S_{\mu}$ is called an $m$-function, 
$$S_{\mu}(z)=m(z)=\langle (J-z)^{-1} e_1, e_1\rangle =(J-z)^{-1}(1, 1).$$ 
Let $J_{k}$ be a Jacobi matrix obtained by removing the top k rows and left-most k columns of $J$ and let $$m_{k}(z) = (J_k-z)^{-1}(1,1) .$$ 
The following relations hold (see \cite{B. Simon}, Theorem 3.2.4)
$$
-\frac{1}{m(z)}=z-\alpha_{0}+\beta_{0}^{2} m_{1}(z)
$$
and 
$$
-\frac{1}{m_1(z)}=z-\alpha_{1}+\beta_{1}^{2} m_{2}(z).
$$
Observing $m_{1}(z)=m_2(z),$ we know $m_1(z)$ satisfies the following equation
$$
\beta_1^{2} m_1(z)^{2}+(z-\alpha_1 ) m_1(z)+1=0.
$$
Therefore, resolving the two equations, we have $$ m(z) =- \frac{2\beta_1^2(z-\alpha_0) - \beta_{0}^{2}\left(z-\alpha_{1}\right)-\beta_{0}^{2} \sqrt{\left(z-\alpha_{1}\right)^{2}-4 \beta_{1}^{2}} }{2\beta_0^4 +2\beta_1^2(z-\alpha_0)^2 +2 \beta_0^2(z-\alpha_0 )(\alpha_1-z )}. $$
When the limit $\lim\limits_{\varepsilon \searrow 0} \operatorname{Im} m(x+i \varepsilon)$ exists and is finite for all $x$ belonging to its natural domain, the spectral measure $\mu$ has density function  (\cite{AnG} , Theorem 2.4.3)
$$
p(x)=\lim _{\varepsilon \searrow 0} \frac{1}{\pi} \operatorname{Im} m(x+i \varepsilon).
$$
Thus, $$ p(x) = \frac{1}{2\pi} \frac{\beta_0^2 \sqrt{ 4 \beta_{1}^{2}-\left(x-\alpha_{1}\right)^{2} }}{ \beta_{0}^{4}+ \beta_{1}^{2}\left(x-\alpha_{0}\right)^{2}+ \beta_{0}^{2}\left(x-\alpha_{0}\right)\left(\alpha_{1}-x\right)}. $$ 
When $$\alpha_0=0, \; \beta_0=\frac{1}{\sqrt{1+\sigma}}, \;  \alpha_1=\frac{\sqrt{\gamma}(1-\sigma)}{1+\sigma}, \; \beta_1=\frac{\sqrt{1+\sigma-\sigma \gamma}}{1+\sigma}, $$
one gets 
$$p(x)=\frac{\sigma\sqrt{\gamma}}{1+\sigma}h_{\gamma, \sigma}\left(\frac{\sigma}{1+\sigma}(\sqrt{\gamma} x+1)\right)=\tilde{h}_{\gamma, \sigma}(x).$$ The proof is then completed. 
\end{proof}

\begin{rmk}\label{Rbeta} It has been proved in \cite{DuEd} that the Wigner semicircle law, Marchenko-Pastur law and the Wachter law correspond to different $\alpha_0, \alpha_1$ and $\beta_0, \beta_1,$ respectively. Here, we give a different proof of Lemma \ref{beta}.  	
\end{rmk}

Now we can present the proof of Theorem \ref{spe}.
\begin{proof}[\bf Proof of Theorem \ref{spe}]
Keep in mind that our aim is to prove that $\mu_n$ converges weakly to $\tilde{\nu}_{\gamma, \sigma}$ almost surely and   
furthermore 
\begin{align}\label{munpcon}
	 \left\langle\mu_{n}, f \right\rangle \rightarrow \langle \widetilde \nu_{\gamma,\sigma}, f \rangle 
\end{align}
almost surely and in $L^{q}$ for all $1 \leq q<\infty$ as $n\to\infty,$ where $f$ is a continuous function with polynomial growth.   

Recall the Jacobi matrix $\widetilde{\mathcal{J}}_{\gamma, \sigma}$ presented in Lemma \ref{Jacobimu} as  
\begin{equation}\label{Jform} \widetilde{\mathcal{J}}_{\gamma,\sigma}=\left(\begin{array}{ccccc}0 & \frac{1}{\sqrt{(1+\sigma)}} & & & \\ \frac{1}{\sqrt{(1+\sigma)}}  & \frac{\sqrt{\gamma}(1-\sigma)}{1+\sigma} &  \sqrt{\frac{1+\sigma-\gamma \sigma}{(1+\sigma)^{2}}}  & & \\ &  \sqrt{\frac{1+\sigma-\gamma \sigma}{(1+\sigma)^{2}}}  & \frac{\sqrt{\gamma}(1-\sigma)}{1+\sigma} &  \sqrt{\frac{1+\sigma-\gamma \sigma}{(1+\sigma)^{2}}}  & \\ & & \ddots & \ddots & \ddots\end{array}\right), \end{equation} 
whose spectral measure is nothing but $\widetilde{\nu}_{\gamma, \sigma}.$ 

By the core of spectral method (see \cite[Th 2.3 and Th 2.9]{Trinh}), once we verify that  
the entrywise limiting matrix of $\widetilde{\mathcal{J}}_{n}$ is $\widetilde{\mathcal{J}}_{\gamma, \sigma},$ in the sense of almost sure convergence,  
  we can conclude that $\mu_n$ converges almost surely to $\widetilde{\nu}_{\gamma, \sigma}$ as $n\to\infty.$ 
 Meanwhile, if the entrywise  limit above also holds with respect to $L^q$ $(1\le q),$ then  the limit \eqref{munpcon} is true with $f$ being a polynomial function. 
\cite[Lem 2.2]{Trinh} helps to get the convergence \eqref{munpcon} when $f$ is a continuous function $f$ of polynomial growth. 

Thereby, it remains to find out the entrywise limit of  
$\widetilde{J}_n$ is $\widetilde{J}_{\gamma, \sigma}$ as $n\to\infty,$ with respect to both the almost sure convergence  and the convergence in $L^q$$(1\le q<\infty).$
 
Recalling the definition of the random variables $c_k, s_k$ in \eqref{cksk} \begin{equation*}
	\begin{cases}
c_{k} \sim \operatorname{Beta}\left( \beta^{\prime} \left(p_1-k+1  \right) , \beta^{\prime} \left(p_2-k+1 \right) \right), &1\leq k \leq n ; \\
s_{k} \sim \operatorname{Beta}\left(\beta^{\prime} \left(n-k \right)  , \beta^{\prime} \left(p_1+p_2-n-k+1  \right)  \right), &1\leq k \leq n-1;
\end{cases}
\end{equation*}
and define
$$
\begin{aligned}
&a_{k}=s_{k-1}\left(1-c_{k-1}\right)+c_k\left(1-s_{k-1}\right), \\
               \end{aligned}
$$
with $c_{0}=s_{0}=0$.
The diagonal entries of the rescaled matrix $\widetilde{\mathcal{J}}_{n}$ are
\begin{equation}\label{jnak}
   \widetilde{\mathcal{J}}_{n}(k, k)  = \frac{p_1+p_2}{\sqrt{np_1}} s_{k-1}\left(1-c_{k-1}-c_k\right)+ \frac{p_1+p_2}{\sqrt{np_1}}c_k -\frac{p_1}{\sqrt{np_1}} 
\end{equation}
and the non negative sub-diagonal entries of $\widetilde{\mathcal{J}}_{n}$ satisfying 
\begin{equation}\label{jnbk}
\widetilde{\mathcal{J}}^2_{n}(k, k+1)  =  \frac{(p_1+p_2)^2}{n p_1}c_k(1-c_k)s_k(1-s_{k-1}). \end{equation}
Since both $s_k$ and $c_k$ satisfy the conditions in Lemma \ref{beta}, it follows from Lemma \ref{beta} and the condition {\bf H} that 
$$\aligned \frac{p_1+p_2}{\sqrt{np_1}}\left(c_k-\frac{p_1}{p_1+p_2}\right)&=\frac{p_1+p_2}{\sqrt{np_1}}\left(c_k-\frac{p_1-k+1}{p_1+p_2-2k+2}\right)-\frac{(p_2-p_1)(k-1)}{(p_1+p_2-2k+2)\sqrt{np_1}}
 \endaligned $$
 converges to $0$ almost surely for any $k\geq 1,$  and 
 $$\aligned \frac{p_1+p_2}{\sqrt{np_1}} s_k=\frac{p_1+p_2}{\sqrt{np_1}} \left(s_k-\frac{n-k}{p_1+p_2-2k+1}\right)+\frac{(p_1+p_2)(n-k)}{\sqrt{np_1}(p_1+p_2-2k+1)}
 \endaligned $$ converges to $\sqrt{\gamma}$
almost surely for any $k\geq 1$ as $n\to\infty,$ 
and similarly
$$\aligned \frac{(p_1+p_2)^2}{np_1} s_k c_k\to 1
 \endaligned $$
 almost surely as $n\to\infty$ for any $k\geq 1.$
 And also, 
$$c_k\to \frac{\sigma}{1+\sigma} \quad \text{and} \quad  s_k\to \frac{\sigma\gamma}{1+\sigma}$$ 
almost surely for any $k\geq 1$ as $n\to\infty.$ 
Putting these limits into \eqref{jnak}, we know 
$$\widetilde{\mathcal{J}}(1, 1)\to 0 \quad \text{and}\quad \widetilde{\mathcal{J}}_{n}(k,k)\to \frac{1-\sigma}{1+\sigma} \sqrt{\gamma}$$ almost surely as $n\to\infty$ for any $k\geq 2.$ 
 Meanwhile, plugging corresponding limits into \eqref{jnbk}, we know 
$$\widetilde{\mathcal{J}}^2(1, 2)\to\frac{1}{1+\sigma} \quad \text{and}\quad \widetilde{\mathcal{J}}^2(k, k+1)\to \frac{1+\sigma-\sigma\gamma}{(1+\sigma)^2}$$ for any $k\geq 2$ almost surely as $n\to\infty.$ 
All the convergence above also hold in $L^q$ for $q\geq 1.$
Therefore, the limiting matrix 
is truly the desired one $\widetilde{\mathcal{J}}_{\gamma, \sigma}.$ 

The proof is then completed. 
\end{proof}
\section{Central limit theorem and large deviation for \texorpdfstring{$\mu_n$}{}}
In this section, we offer two limiting behaviors of the spectral measure $\mu_n$ of $\widetilde{\mathcal{J}}_n:$ central limit theorem and large deviation.   
\subsection{Central limit theorem} 
\begin{thm}\label{muclt}
Let $\mu_n$ be the spectral measure of $\widetilde{\mathcal{J}}_{n}$ as in Theorem \ref{spe}. Under the assumption {\bf H},  for each polynomial $p$, we have 
\begin{align*}
	\sqrt{\beta^{\prime}n}\left(\left\langle\mu_{n}, p \right\rangle-\mathbb E \left\langle\mu_{n}, p \right\rangle  \right) {\rightarrow} \mathcal{N}\left(0, \sigma_{p}^{2}\right)
\end{align*}
weakly as $n\to\infty,$ where $\sigma_{p}^{2}$ is a constant.
\end{thm}
\begin{proof}

	Let $p$ be a polynomial of degree $m>0$. Then there is a polynomial of variables at most $2 m$ such that 
$$
\left\langle\mu_{n}, p\right\rangle=p\left(\widetilde{\mathcal{J}}_n \right)(1,1)=f\left(\left( \widetilde{\mathcal{J}}_n(i, i) , \widetilde{\mathcal{J}}_n(i, i+1) \right)_{ 1\leq i \leq m } \right).
$$
We first show that
$$
\sqrt{\beta^{\prime}n}\left(\left\langle\mu_{n}, p \right\rangle- f\left(\left(  \mathbb{E}\widetilde{\mathcal{J}}_{n}(i, i) , \mathbb{E}\widetilde{\mathcal{J}}_{n}(i, i+1) \right)_{ 1\leq i \leq m }  \right) \right) {\rightarrow} \mathcal{N}\left(0, \sigma_{p}^{2}\right).
$$

Set
\begin{align*}
	& \tilde{a}_k^{(n)}:=\sqrt{\beta' n}(\widetilde{\mathcal{J}}_n(k, k)-\mathbb{E}\widetilde{\mathcal{J}}_{n}(k, k)), \\
	&\tilde{b}_k^{(n)}:=\sqrt{\beta' n}(\widetilde{\mathcal{J}}_n(k, k+1)-\mathbb{E}\widetilde{\mathcal{J}}_{n}(k, k+1)).
\end{align*}
As Theorem 2.9. in \cite{Trinh}, 
it is enough to verify the following conditions: 
\begin{enumerate}
\item $\widetilde{\mathcal{J}}_n(k, k)\to \bar{a}_k, \quad \widetilde{\mathcal{J}}_n(k, k+1)\to \bar{b}_k$ almost surely as $n\to\infty.$ 
\item the two sequence of random variables  
 $$\tilde{a}_k^{(n)} \to \eta_k, \quad \tilde{b}_k^{(n)}\to \zeta_k$$ weakly as $n\to\infty$ and the joint weak convergence also holds. Equivalently, for any two sequences $\{c_k\}_{k\geq 1}$ and $\{d_k\}_{k\geq 1},$
 and $m\geq 1,$ 
 it holds that  \begin{align*}
\sum_{k=1}^m \left(c_{i_k} \tilde{a}_{i_k}^{(n)}+d_{i_k} \tilde{b}_{i_k}^{(n)}\right) \longrightarrow \sum_{k=1}^m\left(c_{i_k} \eta_{i_k}+d_{i_k} \zeta_{i_k}\right)
\end{align*}
weakly 
as $n\to\infty$ for all finites subsets $\left(i_1, \ldots, i_m \right)$ of $ \mathbb Z_+^{m}.$
 \item all moments of $\widetilde{\mathcal{J}}_n(k, k)$ and $\widetilde{\mathcal{J}}_n(k, k+1)$ are finite and the convergence in the first item holds in $L^q$ for $q\geq 1.$ 
 \item $\mathbb{E}\eta_k=\mathbb{E}\zeta_k=0,$ and for some $\delta>0,$ 
  $$\sup_{n}\mathbb{E}[|\tilde{a}_k^{(n)}|^{4}]<\infty, \quad \sup_{n}\mathbb{E}[|\tilde{b}_k^{(n)}|^{4}]<\infty.$$ 
\end{enumerate}

The first and the third items are satisfied as in the proof of Theorem \ref{spe} with $$\bar{a}_1=0, \; \bar{b}_1=\frac{1}{\sqrt{1+\sigma}}, \; \bar{a}_k=\frac{\sqrt{\gamma}(1-\sigma)}{1+\sigma},  \; \bar{b}_{k}=\frac{\sqrt{1+\sigma-\sigma\gamma}}{1+\sigma}$$ for $k\geq 2.$ 
Now we verify the second and the fourth items. We need to check  
\begin{equation}\label{limclt1}\sqrt{\beta' n}(\widetilde{\mathcal{J}}_n(k, k)-\mathbb{E}\widetilde{\mathcal{J}}_{n}(k, k)) \xrightarrow{\rm weakly}{} \mathcal{N}(0, \sigma_{1}^2)\end{equation} 
and 
\begin{equation}\label{limclt2}\sqrt{\beta' n}(\widetilde{\mathcal{J}}_n(k, k+1)-\mathbb{E}\widetilde{\mathcal{J}}_{n}(k, k+1)) \xrightarrow{\rm weakly}{} \mathcal{N}(0, \sigma_{2}^2)\end{equation} as $n\to\infty$ for any $k\geq 2.$ 
Recall for $k\geq 2,$ 
$$\aligned \widetilde{\mathcal{J}}_{n}(k, k) & = \frac{p_1+p_2}{\sqrt{np_1}}s_{k-1}\times\left(1-c_{k-1}-c_k\right)+\frac{p_1+p_2}{\sqrt{np_1}}\left(c_k-\frac{p_1}{p_1+p_2}\right). \endaligned $$ 
Let $x_{k, 1} = \frac{p_1+p_2}{\sqrt{np_1}}s_{k} $ and $x_{k, 2} = \frac{p_1+p_2}{\sqrt{np_1}} c_{k}.$
Hence 
\begin{align}
	\widetilde{\mathcal{J}}_{n}(k, k)-\mathbb{E}\widetilde{\mathcal{J}}_{n}(k, k)
= & (x_{k-1, 1}-\mathbb{E} x_{k-1, 1})\left(1-c_{k-1}-c_k\right) - (x_{k-1, 2}-\mathbb{E} x_{k-1, 2}) \; \mathbb{E} x_{k-1, 1} \notag \\
& + \left( x_{k, 2}-\mathbb{E} x_{k, 2}\right) \left(1- \mathbb E s_{k-1} \right) . \label{jnkk}
\end{align}

Since by condition {\bf H}, 
$$\sqrt{\beta' n}\frac{p_1+p_2}{\sqrt{np_1}}\sqrt{\frac{(n-k+1)(p_1+p_2-n-k+2)}{\beta'(p_1+p_2-2k+3)^3}}\to \sqrt{\gamma}\sqrt{1-\frac{\sigma\gamma}{1+\sigma}}, $$
it follows from Lemma \ref{beta} and the distribution of $s_k$, we know 
$$\sqrt{\beta' n} (x_{k, 1}-\mathbb{E}x_{k, 1})\to \mathcal{N}\left(0, \frac{\gamma(1+\sigma-\sigma\gamma)}{1+\sigma}\right) $$ 
and 
\begin{align*}
	& \sup_{n} \mathbb E \left| \sqrt{\beta' n} (x_{k, 1}-\mathbb{E}x_{k, 1})  \right|^4  \\
	= & \sup_{n} (\sqrt{\beta' n}) ^4 \left( \frac{p_1+p_2}{\sqrt{np_1}}\right)^4  \mathbb E \left| s_{k} - \mathbb E s_{k}  \right|^4  \\
	= & \sup_{n} (\sqrt{\beta' n}) ^4 \left( \frac{p_1+p_2}{\sqrt{np_1}}\right)^4  O \left( \frac{n^2(p_1+p_2-n)^2}{\beta'^2(p_1+p_2)^6}  \right)\\
	=& \sup_{n}O \left( \frac{n^2(p_1+p_2-n)^2}{p_1^2(p_1+p_2)^2}  \right) <  \infty  .
\end{align*}

Similarly, we know 
$$\sqrt{\beta' n} (x_{k, 2}-\mathbb{E}x_{k, 2})\to \mathcal{N}\left(0, \frac1{1+\sigma}\right), $$ 
and
$$\sup_{n} \mathbb E \left| \sqrt{\beta' n} (x_{k, 2}-\mathbb{E}x_{k, 2}) \right|^4 = \sup_{n}  \left( \frac{p_2^2}{(p_1+p_2)^2}   \right) < \infty . $$

Also, 
$1-c_{k-1}-c_k\to (1-\sigma)/(1+\sigma) $ almost surely as $n\to\infty$ and $\mathbb{E}x_{k,1}\to\sqrt{\gamma}$ and $1- \mathbb{E}s_{k}\to 1- \gamma \sigma /(1+\sigma)$  for any $k\geq 2.$
Combining the above results with the independence of $s_{k-1}, c_{k-1}$ and $c_k$, Lemma \ref{slutcoro} helps us obtain he weak limit \eqref{limclt1}.

We also get
\begin{align}\label{supan}
	\sup_{çn\geq 1}\mathbb{E} |\tilde{a}_k^{(n)}|^4<\infty
\end{align}
for all $k\geq 1,$  due to
\begin{align*}
\mathbb{E}[|\tilde{a}_k^{(n)}|^{4}]  \leq & 27 \mathbb E \left| \sqrt{\beta' n} (x_{k-1, 1}-\mathbb{E}x_{k-1, 1})  \right|^4 \mathbb E\left( \left(1-c_{k-1}-c_k\right)^4 \right)  \\
& + 27 \mathbb E \left| \sqrt{\beta' n} (x_{k-1, 2}-\mathbb{E}x_{k-1, 2})  \right|^4 \left(\mathbb E x_{k-1,1} \right)^4 \\
& +  27 \mathbb E \left| \sqrt{\beta' n} (x_{k, 2}-\mathbb{E}x_{k, 2})  \right|^4 \left( 1- \mathbb Es_{k-1}\right)^4 . 
\end{align*}

Now we work on \eqref{limclt2}. We see that
\begin{align}
&\quad \widetilde{\mathcal{J}}_{n}(k, k+1)-\mathbb{E}\widetilde{\mathcal{J}}_{n}(k, k+1)=\frac{\widetilde{\mathcal{J}}^2_{n}(k, k+1)-(\mathbb{E} \widetilde{\mathcal{J}}_{n}(k, k+1))^2}{\widetilde{\mathcal{J}}_{n}(k, k+1)+\mathbb{E}\widetilde{\mathcal{J}}_{n}(k, k+1)}.
\end{align}
     Thus for the verification of \eqref{limclt2}, we need to prove three things: 
\begin{equation}\label{cltnew1}
\sqrt{\beta' n}\left(\widetilde{\mathcal{J}}^2_{n}(k, k+1)-\mathbb{E} \widetilde{\mathcal{J}}^2_{n}(k, k+1)\right)	\xrightarrow{\rm weakly}{} \mathcal{N}(0, \sigma_3^2), \quad n\to\infty,
\end{equation} 
and  
\begin{equation}\label{cltnew1.1}
\sup_{n} \mathbb E \left| \sqrt{\beta' n}\left(\widetilde{\mathcal{J}}^2_{n}(k, k+1)-\mathbb{E} \widetilde{\mathcal{J}}^2_{n}(k, k+1)\right) \right|^4 < \infty,
\end{equation} 
and 
\begin{equation}\label{cltnew2}
\sqrt{\beta' n}\left(\mathbb{E}\widetilde{\mathcal{J}}^2_{n}(k, k+1)-(\mathbb{E} \widetilde{\mathcal{J}}_{n}(k, k+1))^2\right)	\to 0, \quad n\to\infty.
\end{equation}
Recall for $k \geq 2,$
$$\widetilde{\mathcal{J}}_n^2(k, k+1)  = \frac{\left(p_1+p_2\right)^2}{n p_1} c_k\left(1-c_k\right) s_k\left(1-s_{k-1}\right).$$
Let $y_{k,1}= \frac{p_1+p_2}{n} s_k$ and $y_{k,1} = \frac{p_1+p_2}{p_1}c_k. $ Owing to 
\begin{align}\label{jnkk1}
	& \widetilde{\mathcal{J}}_n^2(k, k+1)-\mathbb{E} \widetilde{\mathcal{J}}_n^2(k, k+1) \\
	= & (1-s_{k-1})(1-c_k- \mathbb E c_k) y_{k,1} \cdot \left( y_{k,2} - \mathbb E y_{k,2} \right) 
	 + \mathbb E y_{k, 2}\left(1 - \mathbb E c_k \right)\left(1- s_{k-1}\right) \cdot \left( y_{k, 1} - \mathbb E y_{k, 1} \right) \notag \\
	  & - \mathbb E y_{k, 2} \mathbb E s_k \left( 1 - \mathbb E c_k \right)\left( y_{k-1, 1} - \mathbb E y_{k-1, 1}  \right) 
	  + \mathbb E y_{k, 1} \left(1 - \mathbb E s_{k-1} \right)\mathbb E \left( \left(y_{k,2} - \mathbb E y_{k,2} \right)\left( c_k + \mathbb E c_k \right) \right), \notag
\end{align}  
 the verifications of \eqref{cltnew1} and \eqref{cltnew1.1}  are similar to those of \eqref{limclt1} and \eqref{supan} and omitted here. 
We check the limit \eqref{cltnew2} now.  
If we write
	\begin{align*}
		&  \sqrt{\beta'n} \left( \left(\mathbb{E}\widetilde{\mathcal{J}}_{n}(k, k+1)  \right)^2-\mathbb{E}\widetilde{\mathcal{J}}^2_{n}(k, k+1) \right)  \\
		=&  \mathbb E \left( \sqrt{\beta'n}\left( \widetilde{\mathcal{J}}^2_{n}(k, k+1)-\mathbb{E}\widetilde{\mathcal{J}}^2_{n}(k, k+1)   \right) \times  \frac{\mathbb E \widetilde{\mathcal{J}}_{n}(k, k+1) + \sqrt{\mathbb{E}\widetilde{\mathcal{J}}^2_{n}(k, k+1)}   }{\widetilde{\mathcal{J}}_{n}(k, k+1) + \sqrt{\mathbb{E}\widetilde{\mathcal{J}}^2_{n}(k, k+1)} } \right) \\
		=: & \mathbb E h_n  
	\end{align*}
	Under conditions (1)  and \eqref{cltnew1} and \eqref{cltnew1.1},  we know $$ \frac{\mathbb E \widetilde{\mathcal{J}}_{n}(k, k+1) + \sqrt{\mathbb{E}\widetilde{\mathcal{J}}^2_{n}(k, k+1)}   }{\widetilde{\mathcal{J}}_{n}(k, k+1) + \sqrt{\mathbb{E}\widetilde{\mathcal{J}}^2_{n}(k, k+1)} } \to 1 ,\quad n \rightarrow \infty   $$ almost surely  and $h_n$ converges weakly to a normal  distribution with zero mean and $\sup_n  \mathbb E |h_n|^4 < \infty, $ so \eqref{cltnew2}  holds.

  We see clearly that 
  $$\aligned \mathbb{E}|\tilde{b}_{k}^{(n)}|^4
  \le & 8 \mathbb E \left|  \frac{\sqrt{\beta' n}\left(\widetilde{\mathcal{J}}^2_{n}(k, k+1)-\mathbb{E} \mathcal{J}^2_{n}(k, k+1)\right)}{\widetilde{\mathcal{J}}_{n}(k, k+1)+\mathbb{E}\widetilde{\mathcal{J}}_{n}(k, k+1)}  \right|^4 \\
  & + 8 \mathbb E \left|  \frac{\mathbb{E} \mathcal{J}^2_{n}(k, k+1)-(\mathbb{E} \mathcal{J}_{n}(k, k+1))^2}{\widetilde{\mathcal{J}}_{n}(k, k+1)+\mathbb{E}\widetilde{\mathcal{J}}_{n}(k, k+1)}  \right|^4. 
  \endaligned $$ 
The condition (1) and  \eqref{cltnew1.1} and \eqref{cltnew2}  imply 
$$\sup_{n\geq 1}\mathbb{E} |\tilde{b}_k^{(n)}|^4<\infty$$ for all $k\geq 1.$ 
 
 Combining \eqref{jnkk}, \eqref{jnkk1} with the independence of $s_k$ and $c_k$, Lemma \ref{slutcoro} ensured the joint weak convergence holds.
 
At last, we show \begin{align*}
	\sqrt{\beta^{\prime}n}\left(\left\langle\mu_{n}, p \right\rangle-\mathbb E \left\langle\mu_{n}, p \right\rangle  \right) {\rightarrow} \mathcal{N}\left(0, \sigma_{p}^{2}\right)
\end{align*}
weakly as $n\to\infty.$  It suffices to prove that  \begin{equation}\label{flimit}
	\sqrt{\beta^{\prime}n}\left( f\left(\left(  \mathbb{E}\widetilde{\mathcal{J}}_{n}(i, i) , \mathbb{E}\widetilde{\mathcal{J}}_{n}(i, i+1) \right)_{ 1\leq i \leq m } \right)     - \mathbb E f\left(\left( \widetilde{\mathcal{J}}_n(i, i) , \widetilde{\mathcal{J}}_n(i, i+1) \right)_{ 1\leq i \leq m }  \right) \right) \to 0 \end{equation} as $n\to\infty$. 

Applying the Taylor expansion, we know 
\begin{align*}
 & \sqrt{\beta^{\prime}n}\left( f\left(\left( \mathbb{E}\widetilde{\mathcal{J}}_{n}(i, i) , \mathbb{E}\widetilde{\mathcal{J}}_{n}(i, i+1)  \right)_{ 1\leq i \leq m }  \right)     - \mathbb E f\left(\left( \widetilde{\mathcal{J}}_n(i, i) , \widetilde{\mathcal{J}}_n(i, i+1)  \right)_{ 1\leq i \leq m }   \right) \right)   \\
 =& \sqrt{\beta^{\prime}n} \mathbb E \left( f\left(\left(  \mathbb{E}\widetilde{\mathcal{J}}_{n}(i, i) , \mathbb{E}\widetilde{\mathcal{J}}_{n}(i, i+1)  \right)_{ 1\leq i \leq m }  \right)     - f\left(\left( \widetilde{\mathcal{J}}_n(i, i) , \widetilde{\mathcal{J}}_n(i, i+1)  \right)_{ 1\leq i \leq m }   \right) \right)\\
 = &   \sqrt{\beta^{\prime}n} \mathbb E \left( \sum_{\sum_{i=1}^{m}\left(\alpha_{i}+\beta_{i}\right) \geq 1}c(\alpha, \beta) \prod_{i=1}^{m}\left(\frac{\tilde{a}_i^{(n)}}{\sqrt{\beta' n}} \right)^{\alpha_{i}}\left(\frac{\tilde{b}_i^{(n)}}{\sqrt{\beta' n}} \right)^{\beta_{i}} \right),
\end{align*}
where $\left\{\alpha_{i}\right\}$ and $\left\{\beta_{i}\right\}$ are non negative integers and $\sum_{i=1}^{m}\left(\alpha_{i}+\beta_{i}\right) \geq 1$ and the sum above consists of finitely many terms since $f$ is a polynomial function.  On the one hand, the term corresponding to $\sum_{i=1}^{m}\left(\alpha_{i}+\beta_{i}\right) = 1$ multiplied by $\sqrt{\beta' n} $ converges to 0 in probability, which and the fourth condition guarantee that 
\begin{align*}
 \sqrt{\beta^{\prime}n} \;\mathbb E \left(c(\alpha, \beta) \prod_{i=1}^{m}\left(\frac{\tilde{a}_i^{(n)}}{\sqrt{\beta' n}} \right)^{\alpha_{i}}\left(\frac{\tilde{b}_i^{(n)}}{\sqrt{\beta' n}} \right)^{\beta_{i}}  \right)	  
	\to  0.
\end{align*}
On the other hand, Lemma \ref{beta} ensures that the term related to $\sum_{i=1}^{m}\left(\alpha_{i}+\beta_{i}\right) \geq 2$ multiplied by $\sqrt{\beta' n} $ converges to 0 in $L^q$. Finally, we get 
the desired limit \eqref{flimit}.

The proof is completed. 
\end{proof}

\subsection{Large deviation of \texorpdfstring{$\mu_n$}{}}
We first recall the definition of the large deviation principle. Let $U$ be a topological Hausdorff space with Borel $\sigma$ -algebra $\mathcal{B}(U)$ and $\mathcal{I}: U \rightarrow[0, \infty]$ a lower semicontinuous function. We say that a sequence $\left(P_{n}\right)_{n\geq 1}$ of probability measures on $(U, \mathcal{B}(U))$ satisfies the large deviation principle (LDP) with speed $a_{n}$ and rate function $\mathcal{I}$ if:

\begin{enumerate}
	\item For all closed sets $F \subset U,$ 
$$
\varlimsup _{n \rightarrow \infty} \frac{1}{a_{n}} \log P_{n}(F) \leq-\inf _{x \in F} \mathcal{I}(x).
$$
	\item For all open sets $O \subset U,$ 
$$
\varliminf _{n \rightarrow \infty} \frac{1}{a_{n}} \log P_{n}(O) \geq-\inf _{x \in O} \mathcal{I}(x).
$$
\end{enumerate}

In the following, we say that a random variable is $\operatorname{Gamma}(a)$ distribution if it has the density
$$
\frac{x^{a-1}}{\Gamma(a)} e^{-x} \mathbf{1}_{\{x>0\}}
$$
with $a>0$. The logarithmic moment generating function of the $\operatorname{Gamma}(a)$ distribution is $$\Lambda(t):=\log \mathbb E\left(e^{tx}\right)  =-a \log (1-t)$$(for $t<1)$ with Fenchel-Legendre transform
 $$ \mathcal I(x):=\sup_{t \in R } \{tx-\Lambda(t)\} =a g\left( x/a\right),$$ where the function $g$ is defined by
     \begin{equation}\label{defg}
g(x)=
\begin{cases}
x-\log x-1, & x >0;\\
\infty, & \text{otherwise.}
\end{cases}
\end{equation} 

Before the statement of Jacobi mapping between tridiagonal matrices and its spectral measures, we first introduce the topology on moments spaces, where the spectral measures belong.  
 
{\bf Topology on moments spaces.} Given $I$ a subset of $\mathbb{R}.$ Let $\mathcal{M}^1(I)$ be the collection of all probability measures on the interval $I$ and $$\mathcal{M}_m^1(I):=\left\{\mu\in \mathcal{M}^1(I): \; \left\langle\mu, x^{k}\right\rangle<+\infty, \; \forall \,k\geq 1\right\}.$$ 
 For any $\mu \in \mathcal{M}_m^1(I),$ we set
$$\mathbf{m}(\mu)=\left(\left\langle\mu, x^{k}\right\rangle\right)_{k \geq 1}, $$
which is in $\mathbb{R}^{\mathbb{N}}.$ 
We endow $\mathcal{M}_m^1(I)$ with the distance of convergence of moments:
$$
d(\mu, \nu)=\sum_{k=1}^{\infty} 2^{-k} \frac{\left|\left\langle\mu, x^{k}\right\rangle-\left\langle\nu, x^{k}\right\rangle\right|}{1+\left|\left\langle\mu, x^{k}\right\rangle-\left\langle\nu, x^{k}\right\rangle\right|}.
$$
Denoted by $\mathcal{M}_{m, d}^1(I)$ the subset of $\mathcal{M}_m^1(I),$ whose elements are uniquely determined by their moments. The mapping $\mathbf{m}: \mathcal{M}_{m, d}^1(I)\to \mathbb{R}^{\mathbb{N}}$ is then injective and continuous.

We consider $n \times n$ matrices corresponding to measures supported by $n$ points and semi-infinite matrices corresponding to measures with bounded infinite support.

In the theory of orthogonal polynomials on the real line, given a probability measure $\mu$ on $\mathbb R$ with bounded infinite support (resp. $\mu_n $ with a finite support consisting of $n$ points), the orthonormal polynomials obtained by applying the orthonormalizing Gram-Schmidt procedure to the sequence $1, x, x^2, \ldots$ obeying the recursion relation
\begin{align}\label{pkrec}
x P_{k}(x)=b_{k} P_{k+1}(x)+a_{k} P_{k}(x)+b_{k-1} P_{k-1}(x)	
\end{align}
for $k \geq 1$ (resp. for $1 \leq k \leq n)$ where $\left\{a_k\right\}_{k \geq 1}, \left\{b_k\right\}_{k \geq 1}$ are two sequences of uniformly bounded real numbers such that $a_k\in \mathbb R$ and $b_k \in \mathbb R^+$.  The parameters $\left\{a_k, b_k\right\}_{k \geq 1}$ (resp. $\left\{a_k\right\}_{1 \leq k \leq n }, \left\{b_k\right\}_{1 \leq k \leq n -1}$) are called the Jacobi coefficients with measure $\mu$ (resp. $\mu_{n}$). 

As it is well known (\cite{Simon1}), when $I$ is $\mathbb R$ or $n$ points, \eqref{pkrec} sets up the one-to-one correspondence between $\mu \in \mathcal{M}_{m, d}^1(I)$ and Jacobi coefficients $\left\{a_k, b_k\right\}_{k \geq 1}$ with $a_k\in \mathbb R$ and $b_k \in \mathbb R^+$ and $\sup_n \left( |a_n|+|b_n|\right) < \infty.$ 
It is also well known (see \cite{Chihara}) that the measure $\mu \in  \mathcal{M}_{m, d}^1\left([0, \infty)\right)$ if and only if there exists a sequence $\left\{z_k\right\}_{k \geqslant 1}$ of positive numbers such that the Jacobi coefficients in the recurrence relation \eqref{pkrec} satisfy for all $k \geq 1$ 
\begin{align}\label{decojaco}
 a_k=z_{2 k-2}+z_{2 k-1}, 
 \quad
 b_k^2=z_{2 k-1} z_{2 k},
\end{align}  
where $z_0 = 0.$
In particular, the Marchenko-Pastur law corresponds to $z_{2 k-1}^{\mathrm{MP}}=1$ and $z_{2 k}^{\mathrm{MP}}= \gamma$ for all $k \geq 1$. 
It was further discovered by Wall \cite{Wall} that $\mu \in  \mathcal{M}_{m, d}^1([0,1]) $ if and only if the coefficients $z_k$ form a chain sequence; that is, they can further be decomposed as
\begin{align}\label{defzk}
z_k=p_k \left(  1-p_{k-1} \right)
\end{align}
where $p_0 =1 $ and $0<p_k<1$ for $k \geq 1$. 
Especially, the Wachter law corresponds to $p_{2 k-1}^{\mathrm W}= \frac{\sigma}{1+\sigma} $ and $p_{2 k}^{\mathrm W}= \frac{\gamma\sigma}{1+\sigma} $ for all $k \geq 1$.

In this paper, we consider the case $\mu\in \mathcal{M}_{m, d}^1\left(\left[ -\frac{1}{\sqrt{\gamma}} , \frac{1}{\sqrt{\gamma}\sigma} \right]\right)$ when $\gamma\in [0, 1]$ and $\gamma\sigma\in [0, 1].$ In the cases of $\gamma =0$ or $\gamma\sigma =0$, $+\infty$ will replace $\frac{1}{\sqrt{\gamma}}$ or $\frac{1}{\sqrt{\gamma}\sigma},$ respectively.
For any $\mu\in \mathcal{M}_{m, d}^1\left(\left[ -\frac{1}{\sqrt{\gamma}} , \frac{1}{\sqrt{\gamma}\sigma} \right]\right)$ when $\gamma \in (0,1]$ and $\sigma \gamma \in (0,1],$ the Jacobi coefficients $\left\{a_k, b_k\right\}_{k \geq 1}$ corresponding to $\mu$ can be decomposed uniquely into
\begin{align} \label{newakbk}
\begin{cases}
	a_k =  \frac{1+\sigma}{\sqrt{\gamma}\sigma} \left( p_{2k-2} \left(1-p_{2k-3}\right)+ p_{2k-1} \left(1-p_{2k-2}\right) \right) - \frac{1}{\sqrt{\gamma}} \\
	b_k^2 = \frac{(1+\sigma )^2}{\gamma \sigma^2} p_{2k-1} \left(1- p_{2k-2} \right) p_{2k} \left(1- p_{2k-1} \right)
\end{cases}
\end{align}  
with $p_{-1}=p_0 = 1$ and $0 < p_k < 1 $ for $k \geq 1.$ Set $v_k= \frac{1+\sigma}{\gamma\sigma } p_{2k}$ and $u_k = \frac{(1+\sigma)p_{2k-1} - \sigma}{\sqrt{\gamma}\sigma}$ for $k\geq 0,$ the decomposition 
\eqref{newakbk} can be rewritten as
\begin{align}\label{xyde1}
\begin{cases}
	a_k   =\sqrt{\gamma} \frac{ 1-\sigma}{1+\sigma}  v_{k-1}-  \frac{ \gamma\sigma }{1+\sigma}  v_{k-1}\left(u_{k-1}+u_k \right)+u_{k}, \\ 
 b_k^2   = \frac{1 }{1+ \sigma }v_k\left(1- \sigma \sqrt{\gamma}u_{k}\right) \left(1+\sqrt{\gamma}u_{k}\right)  \left(1- \frac{\gamma \sigma}{1+\sigma} v_{k-1} \right),
\end{cases}
\end{align}
with $v_0 = u_0 =0$ and $ v_k \in \left( 0, \frac{1+\sigma}{\gamma \sigma} \right) $ and $ u_k \in \left(  -\frac{1}{\sqrt{\gamma}} , \frac{1}{\sqrt{\gamma} \sigma} \right).$   
For the cases $\gamma =0$ or $\gamma\sigma =0,$ the expression \eqref{xyde1} still hold via a similar argument. 
Particularly, the probability $\tilde{\nu}_{\gamma, \sigma}$ given by \eqref{dens} corresponds to $u_k\equiv 0$ and $  v_k\equiv 1$ (More details are given in Lemma \ref{Jacobimu}).

Gamboa and Rouault (\cite{Gamboa11}) proved that the spectral measure $\mu_n$ of $\beta$-ensemble satisfies a large deviation principle with a rate function encoded by Jacobi coefficients. 
In the case of Hermite, the spectral measure $\mu_n$ satisfies in $\mathcal{M}_{m, d}^1\left( \mathbb R \right)$ a large deviation principle with speed $\beta n/2$ and rate function 
$$
\mathcal{I}_H(\mu)=\sum_{k=1}^{\infty}\left(\frac{1}{2} a_k^2+g\left(b_k^2\right)\right),
$$
where $a_k$ and $b_k$ are the recursion coefficients of polynomials orthonormal with respect to $\mu, g$ as in \eqref{defg} and both sides may be equal to $+\infty$ simultaneously.
In the case of Laguerre, the spectral measure $\mu_n$ satisfies in $\mathcal{M}_{m, d}^1((0, \infty))$ a large deviation principle with speed $\beta n\gamma/2$ and rate function 
$$
\mathcal{I}_L(\mu) =\sum_{k=1}^{\infty}\left( g\left(z_{2 k-1}\right)+\gamma g\left( \frac{z_{2 k}}{\gamma} \right)\right),
$$
where $z_k$ is given in \eqref{defzk}.
In the case of Jacobi, the spectral measure $\mu_n $ satisfies in $\mathcal{M}_{m, d}^1([0,1])$ a large deviation principle with rate function 
\begin{align*}
	\mathcal{I}_J(\mu)  = & \sum_{k=1}^{\infty} \left(\frac{1+\sigma-\gamma \sigma}{\gamma \sigma}g \left(1 + p_{2k+1} \right) + g \left(1 - p_{2k+1} \right)+ \frac{1}{\gamma \sigma} g \left(1 + p_{2k} \right) +  \frac{1}{\gamma} g \left(1 - p_{2k} \right) \right)
\end{align*}  
with $p_k$ being given in \eqref{defzk}.

For the convenience of the presentation, we first introduce the following functions. For any $\mu\in \mathcal{M}_{m, d}^1\left(\left[ -\frac{1}{\sqrt{\gamma}} , \frac{1}{\sqrt{\gamma}\sigma} \right]\right),$ define
\begin{align}\label{defIrgs}
\begin{aligned}
I_{\gamma, \sigma}(\mu):=& \begin{cases}K_{\gamma, \sigma} ,  & 0<\sigma\gamma\le 1; \\
\sum\limits_{k=1}^\infty \left[ \frac{1}{\gamma} g\left( 1+\sqrt{\gamma}u_k \right) + g(v_k)\right] ,  & \sigma=0,  0< \gamma\le 1; \\
\sum\limits_{k=1}^\infty \left[ \frac{1+\sigma}{2}\left( u_k \right)^2+g(v_k)\right] , & \gamma=0, \sigma\geq 0;\\
  \end{cases}	
\end{aligned}
\end{align} 
and $I_{\gamma, \sigma}=+\infty$ otherwise. 
Here, $\{u_k, v_k\}_{k\geq 1}$  are uniquely decomposed by the Jacobian coefficients corresponding to $\mu$ through \eqref{xyde1} and $g$ is defined as in \eqref{defg} and $$K_{\gamma, \sigma}:= \sum\limits_{k=1}^\infty \left[  \frac{1}{\gamma} g\left(1+ \sqrt{\gamma}u_k\right) +\frac{1}{\gamma\sigma} g\left(1-\sqrt{\gamma}\sigma u_k \right)  +	g(v_k) +  \frac{1-c_{\gamma, \sigma}}{c_{\gamma,\sigma}}  g\left( \frac{1-c_{\gamma,\sigma} v_k}{1-c_{\gamma, \sigma}} \right) \right],$$
where $c_{\gamma, \sigma} = \frac{\gamma \sigma}{1+\sigma}.$

\begin{thm}\label{ldpjacobi}
Let $\mu_{n}$ be the spectral measure of $ \widetilde{\mathcal{J}}_{n} $ as in Theorem \ref{spe}. Under the assumption {\bf H}, $\mu_{n}$ satisfies in $\mathcal{M}_{m, d}^1\left(\left[- \frac{1}{\sqrt{\gamma}} ,  \frac{1}{\sqrt{\gamma}\sigma } \right]\right)$  a large deviation principle  with speed $\beta^{\prime} n$ and good rate function $ I_{\gamma, \sigma}$ defined as above. The unique minimizer of $ I_{\gamma, \sigma}$ is $\tilde{\nu}_{\gamma, \sigma},$ whose density function is  given by \eqref{dens}.
 \end{thm}

Before the proof of Theorem \ref{ldpjacobi}, we state two key lemmas on the large deviation related to Gamma random variables. 

\begin{lem}\label{ldp1} Assume the condition {\bf H} holds. Let $G_1 \sim Gamma(a_n) $ and $G_2 \sim Gamma(b_n) $ be two independent random variables satisfying
$$
\lim_{n\rightarrow \infty} \frac{a_n}{\beta^{\prime}p_1} =\lim_{n\rightarrow \infty} \frac{b_n}{\beta^{\prime}p_2} = 1. 
$$ 
Then, the vector $$ \boldsymbol{M}^{(n)}=\left(\frac{p_2G_1-p_1G_2}{\beta^{\prime }\sqrt{np_1}p_2}, \frac{G_1+G_2}{\beta^{\prime }p_2} ,\frac{G_1}{\beta^{\prime }p_2},\frac{p_1+p_2}{\beta^{\prime }p_1p_2}G_1 \right) $$ satisfies a large deviation principle with speed $\beta^{\prime}n$ and good rate function
 \begin{equation}\label{Igs11}
I_{\gamma, \sigma}^ {\boldsymbol{M}}( \boldsymbol{x})= 
\begin{cases}
 \frac{1}{\gamma} g\left( \frac{x_2+\sqrt{\gamma}x_1}{1+\sigma}\right) +\frac{1}{\gamma\sigma} g\left(\frac{x_2-\sqrt{\gamma}\sigma x_1}{1+\sigma} \right)  +\infty\mathbf 1_{\{\boldsymbol{x} \notin D^ {\boldsymbol{M}}  \} }  , & 0<\sigma \gamma \leq 1 ;\\
\frac{1}{\gamma}g\left( 1+\sqrt{\gamma}x_1\right) + \infty \mathbf 1_{\{\boldsymbol{x} \notin D^ {\boldsymbol{M}} \text{ or } x_2 \neq 1+\sigma\}}   , & \sigma=0,0<\gamma \leq 1 ;\\
\frac{x_1^2}{2(1+\sigma)}	 + \infty \mathbf 1_{\{\boldsymbol{x} \notin D^ {\boldsymbol{M}} \text{ or } x_2 \neq 1+\sigma\}}  , & \gamma=0, \sigma \geq 0;
\end{cases}
\end{equation}
where $D^ {\boldsymbol{M}}:= \{ \boldsymbol{x} \in \mathbb R^4 :   x_3= \frac{\sigma}{1+\sigma}\left(\sqrt{\gamma} x_1+x_2\right),x_4=\sqrt{\gamma}x_1+x_2   \}. $ Furthermore,
$$\lim\limits_{\sigma \rightarrow 0}I_{\gamma, \sigma}^ {\boldsymbol{M}}(x_1, 1+\sigma, x_3, x_4) = I_{\gamma, 0}^ {\boldsymbol{M}}(x_1, 1, x_3, x_4)$$ and
$$\lim\limits_{\gamma \rightarrow 0}I_{\gamma, \sigma}^ {\boldsymbol{M}}(x_1, 1+\sigma, x_3, x_4) = I_{0, \sigma}^ {\boldsymbol{M}}(x_1, 1+\sigma, x_3, x_4) . $$
\end{lem}
\begin{proof}

The logarithm of the moment generating function of $\beta^{\prime}n \boldsymbol{M}^{(n)}$ is given by
\begin{align*}
	& \log \mathbb{E}\left[ \exp\left\{ \beta^{\prime}n \left\langle \boldsymbol t, \boldsymbol{M}^{(n)}\right\rangle\right\} \right] \\
	=&  -a_n \log\left( 1- t_1\sqrt{\frac{n}{p_1} } -\left(t_2+t_3\right) \frac{n}{p_2} -t_4\frac{n(p_1+p_2)}{p_1p_2}    \right) -b_n \log\left( 1+t_1 \sqrt{\frac{n}{p_1}} \frac{p_1}{p_2}-t_2\frac{n}{p_2} \right).
\end{align*}
Thus, for any positive $\gamma $ and $\sigma $ satisfying $0<\sigma \gamma \leq 1$,  we obtain the limit
\begin{align*}
	\Lambda^{\boldsymbol{M}}_{\gamma, \sigma}(\boldsymbol t) :
	=& \lim_{n \rightarrow \infty} ( \beta^{\prime}n)^{-1} \log \mathbb{E}\left[ \exp\left\{ \beta^{\prime}n \left\langle \boldsymbol t, \boldsymbol{M}^{(n)}\right\rangle\right\} \right] \\
	= & -\frac{1}{\gamma} \log \left(1-t_1 \sqrt{\gamma}-\left(t_2+t_3\right) \gamma \sigma-t_4 \gamma(1+\sigma)\right) \\
& -\frac{1}{\gamma \sigma} \log \left(1+t_1 \sqrt{\gamma} \sigma-t_2 \gamma \sigma\right) .
         \end{align*}

Define $\mathcal{D}_{\gamma, \sigma} := \left\{\boldsymbol{t} \in \mathbb{R}^{4}: \Lambda_{\gamma, \sigma}(\boldsymbol{t})<\infty\right\}. $ The following conditions can be verified:
\begin{enumerate}
	\item[(1)] the origin belongs to the interior of  $\mathcal{D}_{\gamma, \sigma}.$
	\item[(2)] $\Lambda^{\boldsymbol{M}}_{\gamma, \sigma}$ is lower semicontinuous function and differentiable throughout $\mathcal{D}_{\gamma, \sigma}$.
	\item[(3)] $\Lambda^{\boldsymbol{M}}_{\gamma, \sigma}$ is steep, namely, $\lim\limits_{n \rightarrow \infty}\left|\nabla \Lambda^{\boldsymbol{M}}_{\gamma, \sigma}(\boldsymbol{t_n})\right|=\infty$ whenever $\left\{\boldsymbol{t_n}\right\}$ is a sequence in $\mathcal{D}_{\gamma, \sigma},$ which converges to some boundary point of $\mathcal{D}_{\gamma, \sigma}$.
\end{enumerate}
Then the G{\"a}rtner-Ellis theorem (\cite{A. Dembo}) yields the large deviation principle for $\boldsymbol{M}^{(n)}$ with speed $\beta^{\prime}n$ and good rate function
\begin{align*}
	 I^{\boldsymbol{M}}_{\gamma, \sigma}(\boldsymbol x)=& \sup_{\boldsymbol t\in R^4 }\{ \left\langle \boldsymbol t, \boldsymbol x\right\rangle -\Lambda^{\boldsymbol{M}}_{\gamma, \sigma}(\boldsymbol t) \},\\
	 = &   \frac{1}{\gamma} g\left( \frac{x_2+\sqrt{\gamma}x_1}{1+\sigma}\right) +\frac{1}{\gamma\sigma} g\left(\frac{x_2-\sqrt{\gamma}\sigma x_1}{1+\sigma} \right)  +\infty\mathbf 1_{\{\boldsymbol{x} \notin D^ {\boldsymbol{M}}  \} },
\end{align*}
where $D^ {\boldsymbol{M}}:= \{ \boldsymbol{x} \in \mathbb R^4 :   x_3= \frac{\sigma}{1+\sigma}\left(\sqrt{\gamma}x_1+x_2\right),x_4=\sqrt{\gamma}x_1+x_2   \}. $ Similarly, for cases $ \sigma=0,0<\gamma \leq 1 $ and $\gamma=0, \sigma \geq 0, $ we get the limits respectively
$$\Lambda^{\boldsymbol{M}}_{\gamma, 0}(\boldsymbol{t})= -\frac{1}{\gamma} \log \left( 1-  t_1\sqrt{\gamma} -t_4\gamma \right) -\frac{1}{\sqrt{\gamma}}t_1 +t_2 $$
and 
$$ \Lambda^{\boldsymbol{M}}_{0, \sigma}(\boldsymbol{t})= \frac{1+\sigma}{2} t_1^2+(1+\sigma)t_2 + \sigma t_3 +(1+\sigma) t_4 . $$
Similarly, as for the case $0<\sigma \gamma \leq 1$, these limits imply the large deviation principle of $\boldsymbol{M}^{(n)}$ with speed  $\beta^{\prime}n$ and the good rate function $I_{\gamma, \sigma}^ {\boldsymbol{M}}$ defined in \eqref{Igs11}.
\end{proof}

\begin{lem}\label{ldp2} Assume the condition {\bf H} holds. Let $G_1$ and $G_2$ be two independent random variables with $G_1 \sim Gamma(a_n) $ and $G_2 \sim Gamma(b_n)$ and 
$$
\lim_{n\rightarrow \infty} \frac{a_n}{\beta^{\prime}n} = \lim_{n\rightarrow \infty} \frac{b_n}{\beta^{\prime}(p_1+p_2-n)} = 1. 
$$
Then, the vector $$\boldsymbol{P}^{(n)}=\left(\frac{G_{1}}{\beta^{\prime}n } , \frac{G_1+G_2}{\beta^{\prime }(p_1+p_2 )},\frac{G_1}{\beta^{\prime }(p_1+p_2 )}, \frac{G_1}{\beta^{\prime} \sqrt{np_1}}  \right) $$ satisfies a large deviation principle with speed $\beta^{\prime}n$ and good rate function  
 \begin{equation}\label{Igs2}
I_{\gamma, \sigma}^{\boldsymbol{P} }( \boldsymbol{x})= 
\begin{cases}
 g(x_1) +  \frac{1-c_{\gamma, \sigma}}{c_{\gamma,\sigma}}  g\left( \frac{x_2-c_{\gamma,\sigma} x_1}{1-c_{\gamma, \sigma}} \right)  +\infty\mathbf 1_{\{\boldsymbol{x} \notin D^ {\boldsymbol{P}}  \} }   , & 0<\sigma \gamma \leq 1 ;\\
   g(x_1) + \infty \mathbf 1_{\{\boldsymbol{x} \notin D^ {\boldsymbol{P}} \text{ or } x_2 \neq 1 \}} , & \sigma=0, \; 0<\gamma \leq 1 ;\\
 g(x_1) + \infty \mathbf 1_{\{\boldsymbol{x} \notin D^ {\boldsymbol{P}} \text{ or } x_2 \neq 1\}}, & \gamma=0, \;\sigma \geq 0;
\end{cases}
\end{equation}
where $c_{\gamma, \sigma} := \frac{\gamma\sigma}{1+\sigma} $ and $D^ {\boldsymbol{P}}:= \{ \boldsymbol{x} \in \mathbb R^4 :   x_3= c_{\gamma, \sigma}x_1, x_4 = \sqrt{\gamma}x_1   \}.$  
Furthermore, 
$$ \lim\limits_{\sigma \rightarrow 0}I_{\gamma, \sigma}^{\boldsymbol{P} }(x_1, 1, x_3, x_4)=   I_{\gamma, 0}^{\boldsymbol{P} }(x_1, 1, x_3, x_4)$$ and $$ \lim\limits_{\gamma \rightarrow 0}I_{\gamma, \sigma}^{\boldsymbol{P} }(x_1, 1, x_3, x_4)=   I_{0, \sigma }^{\boldsymbol{P} }(x_1, 1, x_3, x_4).$$

\end{lem}
\begin{proof}
The logarithm of the moment generating function of $\beta^{\prime}n \boldsymbol{P}^{(n)}$ is given by
\begin{align*}
&\log \mathbb{E}\left[\exp \left\{\beta^{\prime} n\left\langle\boldsymbol{t}, \boldsymbol{P}^{(n)}\right\rangle\right\}\right] \\
=& -a_n \log \left(1-   t_1-t_2 \frac{n}{p_1+p_2} -t_3 \frac{n}{p_1+p_2}-t_4 \sqrt{\frac{n}{p_1}} \right) - b_n\log\left( 1-t_2 \frac{n}{p_1+p_2}\right) . 
\end{align*}
Let 
\begin{align*}
	\Lambda^{\boldsymbol{P}}_{\gamma, \sigma}(\boldsymbol t) :=& \lim_{n \rightarrow \infty}( \beta^{\prime}n )^{-1} \log \mathbb{E}\left[ \exp\left\{ \beta^{\prime}n \left\langle \boldsymbol t, \boldsymbol{P}^{(n)}\right\rangle\right\} \right] . \end{align*}
Under the condition {\bf H}, we obtain the following limit
 \begin{equation*}
\Lambda^{\boldsymbol{P}}_{\gamma, \sigma}(\boldsymbol t)= 
\begin{cases}
 -\log \left( 1-  t_1 -(t_2+t_3)c_{\gamma, \sigma}  - t_4\sqrt{\gamma} \right)  -\frac{1-c_{\gamma, \sigma} }{c_{\gamma, \sigma}}\log(1-t_2c_{\gamma, \sigma} ), & 0<\sigma \gamma \leq 1 ;\\
-\log \left( 1-  t_1 -t_4\sqrt{\gamma} \right)  +t_2 , & \sigma=0,0<\gamma \leq 1 ;\\
-\log \left( 1-  t_1 \right)  +t_2, & \gamma=0, \sigma \geq 0.
\end{cases}
\end{equation*}
As in Lemma \ref{ldp1}, this implies the large deviation principle of $\boldsymbol{P}^{(n)}$ with speed  $\beta^{\prime}n$ and the stated rate function $I_{\gamma, \sigma}^{\boldsymbol{P} }$ defined in \eqref{Igs2}.
\end{proof}

\begin{proof}[\bf Proof of Theorem \ref{ldpjacobi}]  
Recall 
$$ \mathbf{m}\left( \mu_n \right) = \left( \widetilde{\mathcal{J}}^r_{n}\left(1, 1\right) \right)_{r     \geq 1} ,$$ which is injective and continuous from $\mathcal{M}_{m, d}^1\left(\left[- \frac{1}{\sqrt{\gamma}} ,  \frac{1}{\sqrt{\gamma}\sigma } \right]\right)$ to $\mathbb{R}^{\mathbb{N}}.$ 
For the large deviation of $\mu_n$ in $\mathcal{M}_{m, d}^1\left(\left[- \frac{1}{\sqrt{\gamma}} ,  \frac{1}{\sqrt{\gamma}\sigma } \right]\right),$  it is enough to obtain that of $\mathbf{m}\left( \mu_n \right)$ in $\mathbb{R}^{\mathbb{N}}.$ Recall the rate function related to $\mu_n$ is $I_{\gamma, \sigma}$ and let $\widetilde{I}$ be the corresponding rate function of $\mathbf{m}(\mu_n).$ Via the argument in \cite{Gamboa11},  for any $\mu\in \mathcal{M}_{m, d}^1\left(\left[- \frac{1}{\sqrt{\gamma}} ,  \frac{1}{\sqrt{\gamma}\sigma } \right]\right)$ and any $(m_1, \cdots)\in \mathbb{R}^{\mathbb{N}}$ such that  $m_i\equiv \langle \mu, x^i\rangle, $ it holds that 
  $$I_{\gamma, \sigma}(\mu)=\widetilde{I}(m_1, \cdots).$$
Now we work on $\mathbf{m}\left( \mu_n \right).$
Notice that there exists a sequence of polynomials $f_r$ of $ r $ variables, such that
$$
 \widetilde{\mathcal{J}}^r_{n}\left(1, 1\right) = f_r\left(\left(\widetilde{\mathcal{J}}_{n}\left(i, i\right),  \widetilde{\mathcal{J}}_{n}\left(j, j+1\right)\right)_{1\leq i \leq \lceil r / 2\rceil, 1\leq j \leq \lfloor r / 2\rfloor}  \right)
$$
for $r \leq 2 n-1$.

The entries of the rescaled tridiagonal matrix $ \widetilde{\mathcal{J}}_{n} $ are the Jacobi coefficients of the measure $\mu_n$, which are given by \eqref{jnak} and \eqref{jnbk}. In particular, $\left\{ \widetilde{\mathcal{J}}_{n}\left(k, k\right),  \widetilde{\mathcal{J}}_{n}\left(k, k+1\right) \right\}_{k \geq 1}$  are continuous functions of  $$ \boldsymbol{ \bar C}^{(k, n)}:= \left( c_k, \frac{p_1+p_2}{p_1}c_k, \frac{(p_1+p_2)c_k-p_1}{\sqrt{np_1}} \right)$$ and$$ \boldsymbol{ \bar S}^{(k, n)}:= \left(s_k, \frac{p_1+p_2}{\sqrt{np_1}}s_k, \frac{p_1+p_2}{n} s_k \right). $$
Recall $$ c_{k} \sim \operatorname{Beta}\left( \beta^{\prime} \left(p_1-k+1  \right) , \beta^{\prime} \left(p_2-k+1 \right) \right), 1\leq k \leq n . $$ Then $\boldsymbol{ \bar C}^{(k, n)}\mathbf 1_{\{k \leq n\}}$ has the same distribution as
$$ \psi\left(\frac{p_2G_1-p_1G_2 }{\beta^{\prime} \sqrt{np_1}p_2 },\frac{G_1+G_2}{\beta^{\prime}p_2}, \frac{G_1}{\beta^{\prime}p_2}, \frac{(p_1+p_2)G_1}{\beta^{\prime}p_1p_2} \right)\mathbf 1_{\{k \leq n\}} = : \psi \left(  \boldsymbol{M}^{(k, n)} \right) \mathbf 1_{\{k \leq n\}}, $$
where $G_{1} \sim \operatorname{Gamma}\left(\beta^{\prime}\left(p_1-k+1  \right)\right)$ and $G_{2} \sim \operatorname{Gamma}\left(p_2-k+1 \right)$ are independent and 
\begin{align}\label{defpsi}
	\psi\left(x_1, x_2, x_3, x_4\right)=\left(\frac{x_3}{x_2} ,\frac{x_4}{x_2}, \frac{x_1}{x_2} \right) . 
\end{align}
For $k$ fixed, by Lemma \ref{ldp1} and the contraction principle, $\boldsymbol{ \bar C}^{(k, n)} \mathbf 1_{\{k \leq n\}} $ satisfies the large deviation principle with speed $\beta^{\prime}n $ and good rate function $ I_{\gamma, \sigma}^{\boldsymbol{ \bar C} }$ defined as follows
\begin{align*}
	 I_{\gamma, \sigma}^{\boldsymbol{ \bar C} }(\boldsymbol y) & = \inf \left\{ I_{\gamma, \sigma}^{\boldsymbol{ M}}\left(\boldsymbol x \right) \mid \psi\left(\boldsymbol x \right)=\boldsymbol y , \boldsymbol x \in \mathbb R^4, \boldsymbol y \in \mathbb R^3 \right\} ,
\end{align*}
where $I_{\gamma, \sigma}^{\boldsymbol{ M}},$ defined in \eqref{Igs11}, is the good rate function of the random vectors $ \boldsymbol{M}^{(k, n)}.$ Notice that for any $\boldsymbol y=(y_1, y_2, y_3)\in \mathbb{R}^3$ satisfying \begin{equation}\label{conditionony} y_1=\frac{\sigma}{1+\sigma}(\sqrt{\gamma}y_3+1) \quad \text{and} \quad y_2=\sqrt{\gamma}y_3+1,\end{equation} 
 there is a unique $\boldsymbol{x}\in D^{\boldsymbol M}$ with $x_2=1+\sigma$ such that $\psi\left(\boldsymbol{x}\right)=\boldsymbol y,$ noted by $\psi_{\sigma}^{-}\left(\boldsymbol y  \right).$  That is 
\begin{align}\label{defpsi2}
	\psi_{\sigma}^{-}\left(\boldsymbol y  \right):=\left( (1+\sigma)y_3, 1+\sigma, \sigma(\sqrt{\gamma}y_3+1) , (1+\sigma)(\sqrt{\gamma}y_3+1) \right).
\end{align}
Thus, \begin{align*}
	 I_{\gamma, \sigma}^{\boldsymbol{ \bar C} }(\boldsymbol y) 	 &= I_{\gamma, \sigma}^{\boldsymbol{ M}}\left( \psi_{\sigma}^{-}\left(\boldsymbol y  \right)\right)\end{align*}
	 for $\boldsymbol{y}\in\mathbb{R}^3$ satisfying \eqref{conditionony} and $+\infty$ otherwise. 
By definition, it holds that $$ \boldsymbol{ \bar S}^{(k, n)}=\psi \left(\frac{G_{1}}{\beta^{\prime}n } , \frac{G_1+G_2}{\beta^{\prime }(p_1+p_2 )},\frac{G_1}{\beta^{\prime }(p_1+p_2 )}, \frac{G_1}{\beta^{\prime} \sqrt{np_1}}  \right) $$ with $G_{1} \sim \operatorname{Gamma}\left(\beta^{\prime} \left(n-k \right)\right)$, and $G_{2} \sim \operatorname{Gamma} \left(\beta^{\prime}(p_1+p_2-n-k+1)  \right) $ being independent random variables. For any $k\geq 1$ fixed, by  Lemma \ref{ldp2} and the contraction principle once more, $\boldsymbol{ \bar S}^{(k, n)}\mathbf 1_{\{k \leq n\}} $ satisfies the large deviation principle with speed $\beta^{\prime} n$ and good rate function \begin{align*}
	  I_{\gamma, \sigma}^{\boldsymbol{\bar S} }(\boldsymbol y) 
	 & = I_{\gamma, \sigma}^{\boldsymbol{P} }\left(\phi_{\sigma}^-(\boldsymbol y)\right) \end{align*}
with $$\phi_{\sigma}^-(\boldsymbol y)=\left( y_3, 1, c_{\gamma, \sigma} y_3, \sqrt{\gamma}y_3 \right)$$ 
for $\boldsymbol{y}\in\mathbb{R}^3$ satisfying \begin{equation}\label{conditionony2}y_1=c_{\gamma, \sigma} y_3 \quad \text{and} \quad y_2=\sqrt{\gamma} y_3\end{equation} and $+\infty$ otherwise.

Fix $\ell >1.$ By the independence of  $\{\boldsymbol{ \bar C}^{(i, n)}\}_{1\le i\le l}$ and $\{\boldsymbol{ \bar S}^{(j, n)}\}_{1\le j\le l-1}$, the sequence $\left( \boldsymbol{ \bar C}^{(i, n)},\boldsymbol{ \bar S}^{(j, n)} \right)_{1\leq i \leq \ell, 1\leq j \leq \ell-1}$ satisfies the large deviation principle in $\mathbb{R}^{2 \ell-1}$ with speed $\beta^{\prime} n$ and rate function $$\sum_{k=1}^{\ell} I_{\gamma, \sigma}^{\boldsymbol{ \bar C}}(\boldsymbol  x^{(k)}) +\sum_{k=1}^{\ell-1}  I_{\gamma, \sigma}^{\boldsymbol{ \bar S}}(\boldsymbol y^{(k)}).$$

Since $ \left( \widetilde{\mathcal{J}}^r_{n}\left(1, 1\right) \right)_{ 1\leq r\leq 2\ell-1 } $  is a continuous function of $\left( \boldsymbol{ \bar C}^{(i, n)},\boldsymbol{ \bar S}^{(j, n)} \right)_{1\leq i \leq \ell, 1\leq j \leq \ell-1},$ by the contraction principle, the sequence $ \left( \widetilde{\mathcal{J}}^r_{n}\left(1, 1\right) \right)_{ 1\leq r\leq 2\ell-1 } $ satisfies the large deviation principle in $\mathbb{R}^{2 \ell-1}$ with speed $\beta^{\prime} n$ and rate function $\widetilde{I}_{2 \ell-1}$ defined as follows.
Given $\left(m_i\right)_{1\leq i\leq 2 \ell-1} \in \mathbb R^{2\ell -1}.$
 Notice that there is at most  one tridiagonal matrix $J_{\ell}$ built from $\left(a_i, b_j\right)_{1\le i\le \ell, 1\le j\le \ell-1}$ as in  \eqref{defjn} such that
\begin{align}\label{defmr}
	J_{\ell}^r\left(1, 1\right) =m_r, \quad 1\le r\le 2\ell-1.
\end{align}  
For $\left(m_i\right)_{1\leq i\leq 2 \ell-1}\in\mathbb{R}^{2\ell -1}$ such that$\left(a_i, b_j\right)_{1\le i\le \ell, 1\le j\le \ell-1}$ can be decomposed uniquely into $\left(u_i, v_j\right)_{1\le i\le \ell, 1\le j\le \ell-1}, $ define  
$$
\widetilde{I}_{2 \ell-1}\left(m_1, \ldots, m_{2 \ell-1}\right)= \sum_{k=1}^{\ell} I_{\gamma, \sigma}^{\boldsymbol{ \bar C}}(\boldsymbol  x^{(k)}) +\sum_{k=1}^{\ell-1}  I_{\gamma, \sigma}^{\boldsymbol{ \bar S}}(\boldsymbol y^{(k)}).
$$
Here $\boldsymbol{x}^{(k)}$ satisfy \eqref{conditionony} and $\boldsymbol{y}^{(k)}$ satisfy  \eqref{conditionony2} with $x_3^{(k)}=u_k, 1\le k\le \ell$ and $y_3^{(k)}=v_k, 1\le k\le \ell-1.$ 
Otherwise, $\widetilde{I}_{2 \ell-1}\left(m_1, \ldots, m_{2 \ell-1}\right)=+\infty.$ Applying the Dawson-Gartner theorem, $ \mathbf m\left( \mu_n \right) $ satisfies the large deviation principle in $\mathbb R^{ \mathbb N }$ with speed $\beta^{\prime}n$ and good rate function  
\begin{align*}
	\widetilde{I}(m_1, \cdots)=\sup\{\widetilde{I}_{2l-1}(m_1, \cdots, m_{2\ell-1}):\; \ell \geq 1\}= \sum_{k=1}^\infty I_{\gamma, \sigma}^{\boldsymbol{ \bar C}}(\boldsymbol  x^{(k)}) +\sum_{k=1}^\infty  I_{\gamma, \sigma}^{\boldsymbol{ \bar S}}(\boldsymbol y^{(k)}). 
\end{align*}
Hence, the spectral measure $\mu_{n}$ satisfies in $\mathcal{M}_{m, d}^1\left(\left[- \frac{1}{\sqrt{\gamma}} ,  \frac{1}{\sqrt{\gamma}\sigma } \right]\right)$ a large deviation  with speed $\beta^{\prime} n$ and rate function 
\begin{equation}\label{Igsr1}
I_{\gamma, \sigma} ( \mu ) = \sum_{k=1}^\infty I_{\gamma, \sigma}^{\boldsymbol{ \bar C}}(\boldsymbol  x^{(k)}) +\sum_{k=1}^\infty  I_{\gamma, \sigma}^{\boldsymbol{ \bar S}}(\boldsymbol y^{(k)}),
 \end{equation}
where $\boldsymbol{x}^{(k)}$ satisfy \eqref{conditionony} and $\boldsymbol{y}^{(k)}$ verify \eqref{conditionony2}
and moreover, $(x_3^{(k)})_{k\geq 1}$ and $(y_{3}^{(k)})_{k\geq 1}$ are the unique solutions to the equation \eqref{xyde1} with $(a_k)_{k\geq 1}$ and $(b_k)_{k\geq 1}$ are the corresponding Jacobi matrix coefficients of $\mu$
and $x_1^{(0)}=y_1^{(0)}=y_2^{(0)}=0.$

	The function $g(x)= x -\log x -1 $ attains its minimal value $0$ at $x=1$. Therefore, the minimum of $I_{\gamma, \sigma}$ is attained at $x_3^{(k)}=0$ and  $y_3^{(k)}=1$ for all $k\geq 1.$ By \eqref{xyde1}, the corresponding Jacobi coefficients should be 
	$$a_1=0, \; b_1=\frac{1}{\sqrt{1+\sigma}}, \; a_k=\frac{\sqrt{\gamma}(1-\sigma)}{1+\sigma}, \; b_k=\frac{\sqrt{1+\sigma-\sigma \gamma}}{1+\sigma} $$
	for $k\geq2$. Therefore, the spectral measure of the related Jacobi matrix is $\tilde{\nu}_{\gamma, \sigma}$, whose density function is given by \eqref{dens}.

\end{proof}

\begin{lem}\label{slutcoro}
Given an integer $m\geq 1.$ Let $\left\{\xi_{n,1}\right\}, \ldots, \left\{\xi_{n, m} \right\} $  be $m$ sequences of independent random variables such that for $1\leq i \leq m,$ $ \xi_{n, i} \to \xi_{ i}$ weakly, as  $ n \to \infty.$
Let $\left\{\eta_{n,1}\right\}, \ldots, \left\{\eta_{n, m}\right\}$  be $m$ sequences of random variables such that for $1\leq i \leq m,$ $ \eta_{n, i} \rightarrow a_{ i} $ in probability   as $n \to \infty,$ where $a_1, \ldots, a_m$ are constants. Then, 
$$
\sum_{i=1}^m \eta_{n, i}\xi_{n, i} \to \sum_{i=1}^m a_{i}\xi_{i} 
$$ 
weakly as $n \to \infty.$
\end{lem}
\begin{proof}
	By using Slutsky's theorem, for $1\leq i \leq m,$ $ \left(\eta_{n, i}- a_{i}\right) \xi_{n, i} \to 0 $ in probability   as $n \to \infty.$  This implies that $$\sum_{i=1}^m \left(\eta_{n, i}- a_{i}\right) \xi_{n, i} \to 0 $$ in probability   as $n \to \infty,$ since $m$ is a fixed integer. It follows from the independence of radom variable sequences $\{\xi_{n, i}\}_{n\geq 1, 1\le i\le m}$ that $$\sum_{i=1}^m a_i \xi_{n, i}  \xrightarrow{\rm weakly}{} \sum_{i=1}^m a_i \xi_{ i} , \text{ as } n\to \infty.$$
By using Slutsky's theorem again, we have 
$$\sum_{i=1}^m \left(\eta_{n, i}- a_{i}\right) \xi_{n, i} +\sum_{i=1}^m a_i \xi_{n, i}  \xrightarrow{\rm weakly}{} \sum_{i=1}^m a_i \xi_{ i} .$$
This completes the proof.
\end{proof}

\bigskip
\subsection*{Acknowledgments} The authors would like to thank Professor Yu-Hui Zhang for his valuable conversations. 
\bigskip
\subsection*{Data Availability} All data generated or analyzed during this study are included in this published article.
\bigskip
\subsection*{Conflict of interest} The authors declare that they have no conflict of interest.

\end{document}